\documentclass[12pt]{amsart}
\usepackage{amsmath,amsthm,amsfonts,amssymb}

\newtheorem{theorem}{Theorem}[section] 
 
\newtheorem{lemma}[theorem]{Lemma}

\newtheorem{prop}[theorem]{Proposition}
\newtheorem{remark}[theorem]{Remark}

\numberwithin{equation}{section}

\usepackage{xcolor}

\RequirePackage[colorlinks,citecolor=blue,urlcolor=blue]{hyperref}
\newcommand{\ol}{\overline}
\newcommand{\eps}{\varepsilon}

\begin{document}

\title{The radius of a self-repelling star polymer}
\author{Carl Mueller}
\address{Carl Mueller
\\Department of Mathematics
\\University of Rochester
\\Rochester, NY  14627}
\email{carl.e.mueller@rochester.edu}
\author{Eyal Neuman}
\address{Eyal Neuman
\\Department of Mathematics
\\Imperial College
\\London, UK}
\email{e.neumann@imperial.ac.uk}
\thanks{CM was partially supported by Simons grant 513424.} 

\begin{abstract} 
We study the effective radius of weakly self-avoiding star polymers in one, 
two, and three dimensions.  Our model includes $N$ Brownian motions up to 
time $T$, started at the origin and subject to exponential penalization 
based on the amount of time they spend close to each other, or close to 
themselves.  The effective radius measures the typical distance from the 
origin.  Our main result gives estimates for the effective radius where in 
two and three dimensions we impose the restriction that $T \leq N$.  One of 
the highlights of our results is that in two dimensions, we find that the 
radius is proportional to $T^{3/4}$, up to logarithmic corrections. Our result may shed light on the 
well-known conjecture that for a single self-avoiding random walk in two 
dimensions, the end-to-end distance up to time $T$ is roughly $T^{3/4}$.  
\end{abstract}

\maketitle

\section{Introduction}
\label{section:introduction}

Random polymer models have caught the imagination of many mathematicians.  
Polymers are all around us, and their behavior presents 
attractive mathematical challenges, many of which still defy solution.  See 
Doi and Edwards \cite{DE09} for a wide-ranging treatment from the physical 
point of view, and Madras and Slade \cite{MS13}, den Hollander \cite{dH09}, 
Giacomin \cite{Gia07}, and Bauerschmidt et.  al. \cite{BDGS12} for rigorous 
mathematical results.   Van der Hofstad and K\"onig \cite{vdHK2001} and 
van der Hofstad, den Hollander, and K\"onig \cite{HHK97}, \cite{HHK03} 
discuss the one-dimensional situation for the related Edwards model.  
Bauerschmidt, Brydges, and Slade \cite{BBS19} have recently developed a 
rigorous version of the renormalization group and applied it to the 
four-dimensional situation.  

In continuous time, we can view a polymer as a Brownian motion with 
penalization for self-intersections.  Here the time parameter represents 
distance along the polymer.  For $T>0$, let $(B_t)_{t\in[0,T]}$ be a 
standard Brownian motion in $\mathbb{R}^d$, defined on a filtered 
probability space 
$(\Omega,(\mathcal{F}_t)_{t\in[0,T]},\mathcal{F},P_T)$.  For a 
probability measure $P$, we write $E^P$ for the corresponding 
expectation.  Since Brownian motion does not have self-intersections in high 
dimensions, we study close approaches instead. For any $r>0$ let 
$\mathbf{B}_r(x)\subset\mathbb{R}^d$ be the open ball of radius $r$ centered at 
$x\in\mathbb{R}^d$, and define
\begin{equation*}
L_T(x) = L_{d,T}(x) := \int_{0}^{T}\mathbf{1}_{\mathbf{B}_1(x)}(B_t)dt
\end{equation*}
for $T>0$.  For $\beta>0$, the typical penalization term is
\begin{equation*}
\mathcal{E}_T = \mathcal{E}_{d,\beta,T} 
:= \exp\left(-\beta\int_{\mathbb{R}^d}L_T(x)^2dx\right).
\end{equation*}
Then we define the penalized measure as
\begin{equation*}
\begin{split}
Q_T(A) &= Q_{d,\beta,T}(A) := \frac{1}{Z_T}
   E^{P_T}\left[\mathbf{1}_A\mathcal{E}_T\right]  \\
Z_T &= Z_{d,\beta,T} = E^{P_T}\left[\mathcal{E}_T\right] 
\end{split}
\end{equation*}
for $A\in\mathcal{F}$.  For simplicity of notation, we have suppressed 
some of the subscripts.  

With these definitions, we call our process weakly self-avoiding Brownian 
motion.  

Note that all of the above quantities depend implicitly on $d$ and all but 
$P,\mathbf{B},(B_t),L$ depend on $\beta$ as well. For simplicity of notation 
we suppress these dependencies, and we will use similar simplified notation 
throughout the paper.  Furthermore, $C$ will stand for a constant which 
could change from line to line, and may also depend on $d$.  

One of the most important problems about weakly self-avoiding Brownian 
motion is to study the radius of the polymer, often defined as the standard 
deviation of the end-to-end distance,
\begin{equation*}
R_T = R_{d,\beta,T} = \left(E^{Q_T}\left[B_T^2\right]\right)^{1/2}.
\end{equation*}
A well-known conjecture from physics states that there exists a scaling 
exponent $\nu_d$ not depending on $\beta$ such that, in some unspecified sense,
\begin{equation*}
R_{d,\beta,T} \approx T^{\nu_d}
\end{equation*}
as $T\to\infty$.  All that is rigorously known is that $\nu_1=1$ 
(Bolthausen \cite{Bol90}, Greven and den Hollander \cite{GdH93}) 
and that $\nu_d=1/2$ for $d\ge5$ (Hara and Slade \cite{HS92:RMP}).  
It is believed that $\nu_2=3/4$, and there are 
connections to SLE$_{8/3}$ (see Lawler, Schramm, and Werner \cite{LSW04}). 
This conjecture has received enormous 
attention, and Duminil-Copin and Smirnov \cite{DCS12}, page 9, write
``The derivation of these exponents seems to be one of the most challenging 
problems in probability.''  In \cite{BDGS12} Section 1.5.2, we learn that 
``Almost nothing is known rigorously about $\nu$ in dimensions 2, 3, 4. 
It is an open problem to show that the mean-square displacement grows at 
least as rapidly as simple random walk, and grows more slowly than 
ballistically''.  In our language, this means that for $d\in\{2,3,4\}$ it has 
not been proved that $\nu_d\ge1/2$ or $\nu_d<1$. 
One of the highlights of the present work is that we do obtain the exponent
$3/4$ in $d=2$, see Theorem \ref{thm-radius}.  Of course this does not settle the above 
conjecture, and our methods are unlikely to settle it.

In the real world, many polymers are branched and do not consist of a single 
strand.  van der Hofstad and K\"onig \cite{vdHK2001}, (Section 3.1 pages 16 
- 18) give a short discussion of branched polymers taking values in 
$\mathbb{R}$,  and present a conjecture for the growth of the radius.  
As far as we know, the conjecture is still open.  Slade and van der Hofstad
\cite{SH03} use the lace expansion to study the radius for branched polymers 
in high dimensions, and show that they behave as if there were no 
self-avoidance.  

Clearly, self-avoiding polymers present difficult challenges in low 
dimensions. In this paper we focus on a related problem, the case of star 
polymers.  Star polymers are polymers joined at the point $t=0$, and 
there is an extensive physics literature about them. See the seminal paper 
of Daoud and Cotton \cite{DC82}, and for more recent work see
\cite{AKPI12} and \cite{TVM21}.

Since self-avoiding polymers present difficult challenges in low dimensions, 
we focus on the case of star polymers, which are not too different from 
random walks.  Star polymers are polymers joined at the point $t=0$, and 
there is an extensive physics literature about them. See the seminal paper 
of Daoud and Cotton \cite{DC82}, and for more recent work see
\cite{AKPI12} and \cite{TVM21}.

We now give a brief overview of the results from \cite{DC82} 
which are relevant to this paper, and ask the reader to keep in mind that 
these results come from mathematically non-rigorous arguments.  

First, the authors formulate a notion of radius relevant to star polymers.  
Then, for a given value of $T$, they discuss two regions:
\begin{enumerate}
\item The unswollen region closer to the origin, where many paths overlap.  
\item The swollen region far from the origin, where pairs of paths rarely 
overlap.  
\end{enumerate}
Our results deal with region (1). One of the principal results of 
\cite{DC82} is equation (19), which states that for very long chains, or for 
high temperatures,
\begin{equation} \label{eq:DC-result}
R\approx N^{3/5}v^{1/5}f^{1/5}\ell.    
\end{equation}
To aid the comparison with our results, we give a dictionary for the notation
in the above equation.  

\medskip
\begin{tabular}{|l|l|l|}\hline
Quantity & Our Notation & From \cite{DC82} \\\hline
Number of branches &  $N$  & $f$  \\
Length along the polymer &  $T$  & $N$  \\
Self-avoidance parameter &  $\beta$  & $v$  \\
Length of each polymer element & not included & $\ell$  \\ \hline
\end{tabular}
\medskip

\noindent
Translating to our notation, \eqref{eq:DC-result} means, for $d=3$, that
\begin{equation}  \label{eq:physical-result}
R_T \approx \beta^{1/5}N^{1/5}T^{3/5}
\end{equation}
for large values of $T$.  Again, this is a nonrigorous physical result.
Our main result, Theorem \ref{thm-radius}, gives a range of values for $R_T$ 
in $d=3$, that includes \eqref{eq:physical-result}. The physical reasoning in Daoud and Cotton's
paper yields the following conjecture for the two dimensional model,  
\begin{equation} \label{eq:Bish-result}
R\approx \ \beta^{1/4} N^{1/4} T^{3/4}.    
\end{equation}
We refer to Bishop et al. \cite{Bishop83} who studied the two dimensional star polymer model. In Theorem \ref{thm-radius} we verify \eqref{eq:Bish-result} up to a logarithmic constant. 

We will now give rigorous definitions, and redefine the notation $P_T$, 
$Q_T$, $L_T(\cdot)$, $\mathcal{E}_T$, $R_T$ used earlier.  Assume that for each 
$T>0$ we have a filtered probability space 
$(\Omega,\mathcal{F},(\mathcal{F}_t)_{t\in[0,T]},P_T)$, and on this 
space, for $d\ge1$, we have a collection 
$(B^{(k)}_t)_{t\in[0,T];\; k\in\{1,\dots,N\}}$
of independent adapted $\mathbb{R}^d$-valued standard Brownian motions 
started at the origin. Without loss of generality, we can assume that 
$\Omega=(C[0,T])^N$ is canonical path space for the Brownian motions.

We define a weakly self-avoiding model as follows.  For
$N\in\mathbb{N}=\{1,2,\dots\}$, $T>0$, $x\in\mathbb{R}^d$, and 
$\mathbf{B}_1(x)\subset\mathbb{R}^d$ as before, consider the occupation measure
\begin{equation} \label{eq:l-time}
L_T(x) = L_{T,d,N}(x) := \sum_{k=1}^{N}|\{t\in[0,T]: B^{(k)}_t\in\mathbf{B}_1(x)\}|
\end{equation}
where $|S|$ denotes the Lebesgue measure of a Borel set $S\subset\mathbb{R}$.  
Our penalization factor is defined as
\begin{equation*}
\mathcal{E}_{T} := \mathcal{E}_{d,\beta,N,T} 
= \exp\left(-\beta \int_{\mathbb{R}^d}L_T(x)^2dx\right).
\end{equation*}
We define a probability $Q_T=Q_{T,d,N,\beta}$ and a normalizing factor 
$Z_T=Z_{T,d,N,\beta}$ as
\begin{align} \label{eq:Q-def}
Q_T(S) := \frac{1}{Z_T}E^{P_T}\big[\mathbf{1}_S\mathcal{E}_T\big],
&& Z_T := E^{P_T}\big[\mathcal{E}_T\big].
\end{align}
for $S\in\mathcal{F}_T$.  

For any set of real numbers $A= \{a_1,...,a_N\}$ we denote by $\textrm{med}(A)$
the median of the set. We define the radius of the star polymer as follows, 
\begin{equation} \label{rad-def} 
R_T=R_{d,N,T}
:=\textrm{med} \left(\left\{\sup_{t\in[0,T]}|B^{(k)}_t| \, 
        : k=1,...,N\right\}\right).
\end{equation}
Our goal is to study the behavior of $R_T$ under the measure $Q_T$. 

For convenience we have chosen a different definition of radius than 
before.  But there are already several definitions of the radius, for 
example see Fixman \cite{Fix62}.

Our main result is stated in the following theorem.  
 \begin{theorem}  \label{thm-radius}  
There exists positive constants $C_d, c_d$ and $C$ not depending on $N,T$ and $\beta$ such that, 
\begin{itemize}
\item[ \textbf{(i)}] for $d=1$, for all $\beta, N \geq 1$, and $T \geq c_1^{-1}$, 
\begin{equation*}
Q_T\left( c_1\beta^{\frac{1}{3}}N^{\frac{1}{3}} T \leq R_T \leq C_1\beta^{\frac{1}{3}}N^{\frac{1}{3}} T \right) 
\geq 1 - \exp\left(-C \beta^{\frac{2}{3}} N^{\frac{5}{3}} T\right), 
\end{equation*}
\item[ \textbf{(ii)}] for $d=2$, for all $\beta \geq 1$, $N \geq  (c_2^{-4/3} \vee1) $ and $  c_2^{-4/3}\leq T \leq N$, 
\begin{multline*} 
Q_T\left( c_2\beta^{\frac{1}{4}}N^{\frac{1}{4}} T^{\frac{3}{4}} (\log(\beta T))^{-\frac{1}{2}} \leq  R_T \leq C_2\beta^{\frac{1}{4}}N^{\frac{1}{4}} T^{\frac{3}{4}}  (\log(\beta T))^{\frac{1}{2}}\right) 
\\ \geq 1 - \exp\left(-C \beta^{\frac{1}{2}} N^{\frac{3}{2}} T^{\frac{1}{2}} \log(\beta T)\right), 
\end{multline*} 
\item[ \textbf{(iii)}] for $d=3$, for all $\beta  \geq 1$, $N \geq (c_3^{-2}\vee 1)$ and $c_3^{-2} \leq T \leq N$, 
\begin{multline*}
Q_T\left(  c_3\beta^{\frac{1}{6}}N^{\frac{1}{6}} T^{\frac{1}{2}}(\log(\beta T))^{-\frac{1}{3}}  \leq R_T \leq  C_3\beta^{\frac{1}{4}}N^{\frac{1}{4}} T^{\frac{3}{4}}  (\log(\beta T))^{\frac{1}{2}}\right) \\
\geq 1 - \exp\left(-C  \beta^{\frac{1}{2}} N^{\frac{3}{2}} T^{\frac{1}{2}} \log(\beta T)\right).
\end{multline*} 
\end{itemize} 
\end{theorem} 
Note that the upper and lower bounds in part (iii) of Theorem 
\ref{thm-radius} include the physical result given in 
\eqref{eq:physical-result}.

The proof of Theorem \ref{thm-radius} is given in Section \ref{sec-results}. 

\section{strategy} 

Now we outline the strategy of our proof.  

In order to to study our radius $R_T$, we introduce the following events
\begin{equation} \label{eq:def-A}
\begin{split}
A^{(<)}_{T,r} &= \{R_T\le r\},  \\
A^{(>)}_{T,r} &= \{R_T\ge r\}.
\end{split}
\end{equation}
We will show that for appropriate functions $r_1(T),r_2(T)$,
\begin{equation*}
\begin{split}
\lim_{T\to\infty}Q_T\big(A^{(<)}_{T,r_1(T)}\big) &= 0,  \\
\lim_{T\to\infty}Q_T\big(A^{(>)}_{T,r_2(T)}\big) &= 0.
\end{split}
\end{equation*}
If we are given functions $r_i:(0,\infty)\to(0,\infty)$ with $i\in\{1,2\}$, we define
\begin{equation} \label{p-def} 
\begin{split}
q_T^{(<)} &= Q_T(A^{(<)}_{T,r_1(T)})
     = \frac{1}{Z_T}E^{P_T}\Big[\mathbf{1}_{A^{(<)}_{T,r_1(T)}}\mathcal{E}_T\Big],  \\
q_T^{(>)} &= Q_T(A^{(>)}_{T,r_2(T)})
     = \frac{1}{Z_T}E^{P_T}\Big[\mathbf{1}_{A^{(>)}_{T,r_2(T)}}\mathcal{E}_T\Big]. 
\end{split}
\end{equation}
Note that $q_T^{(<)},q_T^{(>)}$ implicitly depend on $r_1,r_2$ and also
$d,N,\beta$. Then
\begin{align} \label{eq:p1-est}
q_T^{(<)} &\le \frac{1}{Z_T}\sup_{\omega\in A^{(<)}_{T,r_1(T)}}\mathcal{E}_T(\omega),  \\
              \label{eq:p2-est}
q_T^{(>)} &\le \frac{1}{Z_T}E^{P_T}\Big[\mathbf{1}_{A^{(>)}_{T,r_2(T)}}\Big] 
   = \frac{1}{Z_T}P_T\big(A^{(>)}_{T,r_2(T)}\big). 
\end{align}

We first consider \eqref{eq:p1-est}.  Now \eqref{rad-def} shows that on
$A^{(<)}_{T,r_1(T)}$, at least $[N/2]$ Brownian motions satisfy 
$\sup_{t\in[0,T]}|B^{(k)}_t|\le r_1(T)$.  Here for any $x \in \mathbb R$, 
$[x]$ is the greatest integer less than or equal to $x$.  On 
$A^{(<)}_{T,r_1(T)}$, let $\{k_1,\dots,k_{[N/2]}\}$ be the first $[N/2]$ 
indices of the Brownian motions satisfying this condition, and define
$L_T^{(\text{med})}$ be the total occupation measure of these Brownian motions.

On $A^{(<)}_{T,r_1(T)}$ we have that $L^{(\text{med})}_T(\cdot)$ is 
supported on $\mathbf{B}_{r_1(T)+1}(0)$ and
$\int_{\mathbb{R}^d}L^{(\text{med})}_T(x)dx=[N/2]T$.  
The Cauchy-Schwarz inequality shows that among nonnegative functions $f$
supported on $\mathbf{B}_{r_1(T)+1}(0)$, such that 
$\int_{\mathbb{R}^d}f(x)dx=[N/2]T$, the minimum of $\int_{\mathbb{R}^d}f(x)^2dx$
is achieved when $f$ equals a constant $K$ on $\mathbf{B}_{r_1(T)+1}(0)$, 
and in that case
\begin{equation*}
K = \frac{[N/2]T}{V_d\cdot (r_1(T)+1)^d}, 
\end{equation*}
where $V_d$ is the volume the unit $d$-dimensional ball. 

So, on $A^{(<)}_{T,r_1}$ we have
\begin{equation*}
\begin{split}
\int_{\mathbb{R}^d}L_T(y)^2dy 
&\ge K^2|\mathbf{B}_{r_1(T)+1}(0)|  \\
&\ge \left(\frac{[N/2]T}{V_d\cdot (r_1(T)+1)^d}\right)^2 V_d \cdot (r_1(T)+1)^d \\
&= C\frac{N^2T^2}{(r_1(T)+1)^d}.
\end{split}
\end{equation*}
Then by the definition of $\mathcal{E}_T$, we have
\begin{equation} \label{eq:est-p-less}
q_T^{(<)} \le \frac{1}{Z_T}\sup_{\omega\in A^{(<)}_{T,r_1(T)}}\mathcal{E}_T(\omega) 
       \le \frac{1}{Z_T}\exp\left(-\beta C\frac{N^2T^2}{(r_1(T)+1)^d}\right).
\end{equation}

Now we turn to \eqref{eq:p2-est}.  In order to bound $q_T^{(>)}$,
we use the following large deviations result.  
\begin{lemma} \label{lemma-ber} 
Given $p\in(0,1)$, let $(X_i)_{i\geq 1}$ be a sequence of independent
Bernoulli random variables with $P(X_i=1)=p$ and $P(X_i=0)=q:=1-p$. 
Define $S_{n}:=\sum_{i=1}^{n}X_i$. For $\alpha\in(p,1)$ we have
\begin{equation*}
P(S_n>\alpha n) \le \left(\frac{(1-\alpha)p}{\alpha q}\right)^{\alpha n}
   \left[q+p\frac{\alpha q}{(1-\alpha)p}\right]^n.  
\end{equation*}
\end{lemma} 
We give a short proof of Lemma \ref{lemma-ber} in the appendix.  
 
Now let $r=r_2(T)$, and define 
\begin{equation*}
X_i=X_i^{(T)}:=\mathbf{1}_{\{\sup_{t\in[0,T]}|B^{(i)}_t|>r_2(T)\}},
\end{equation*}
and $p=p_T:=P(X_i^{(T)}=1)$.  

Assuming that $p_T\in(0,1/2)$ and $\alpha=1/2$, we conclude
\begin{equation} \label{eq:large-dev-half}
\begin{split}
q_T^{(>)}&=\frac{1}{Z_T}P(S_N>N/2) 
  \le \frac{1}{Z_T}\left(\frac{p_T}{1-p_T}\right)^{N/2}
        \left[2(1-p_T)\right]^N  \\
&= \frac{1}{Z_T}2^N[p_T(1-p_T)]^{N/2} 
     \le \frac{1}{Z_T}(4p_T)^{N/2}
\end{split}
\end{equation}
We expect that $p_T$ is close to 0 for large $T$, so not much is lost in the  
final step above. 

The probability that a single Brownian motion exits $\mathbf{B}_{r_2(T)}(0)$ 
by time $T$ is bounded by
\begin{equation} \label{p-n-bnd} 
p_T =  P_T\Big( \sup_{t\in[0,T]}|B^{(k)}_t|>r_2(T)\Big) 
\leq C\frac{1}{r_2(T)}\exp\left(-\frac{r_2(T)^2}{2T}\right), 
\end{equation}
by the reflection principle and standard Brownian estimates.

Continuing, we use \eqref{eq:large-dev-half} and \eqref{p-n-bnd} to get
\begin{equation*} 
q_T^{(>)} \le \frac{1}{Z_T}(4p_T)^{N/2}  
\le \frac{1}{Z_T}\left(\frac{C}{r_2(T)}\right)^{N/2}\exp\left(-\frac{Nr_2(T)^2}{2T}\right).
\end{equation*}
Now assuming that $r_2(T)^2/T>C_0>0$ we can absorb $(C/r_2(T))^{N/2}$ into the 
exponential to get
\begin{equation} \label{eq:est-p-greater}
q_T^{(>)} \le \frac{1}{Z_T}\exp\left(-\frac{CNr_2(T)^2}{T}\right),
\end{equation}
for some constant $C>0$.

From \eqref{rad-def} it follows that on $A^{(>)}_{T,r_2}$, at least 
$[N/2]$ of the Brownian paths exit from $\mathbf{B}_{r_2(T)}(0)$ 
within time $[0,T]$.  

Our next argument is only heuristic, but it allows us to guess a formula for
$R_T$.  Recall that $Q_T$ in \eqref{eq:Q-def} involves the ratio of $P_T$ to $Z_T$.  
We think of $R_T$ as the critical value of $r_1(T)$ and $r_2(T)$ for which both
$q^{(<)}_T,q^{(>)}_T$ are close to $Z_T$.  We believe that if
$r_1(T)\approx r_2(T)\approx R_T$, then $q^{(<)}_T,q^{(>)}_T$ will be close.  
Thus we set $r_1(T)\approx r_2(T)$ and equalize
the powers appearing in \eqref{eq:est-p-less} and \eqref{eq:est-p-greater}.
Following this path, we conclude that (ignoring constants)
\begin{equation*}
\frac{Nr^2_2(T)}{T} \approx \beta  \frac{N^2T^2}{r_1(T)^d}
\end{equation*}
hence setting $r_1(T)\approx r_2(T)\approx R_T$, we get the guess
\begin{equation*}
R_T \approx \beta^{\frac{1}{d+2}} N^{\frac{1}{d+2}} T^{\frac{3}{d+2}}.
\end{equation*}

This leads us to guess that, up to logarithmic factors,
\begin{equation} \label{z-asymp} 
Z_T \approx \beta^{\frac{2}{d+2}} N^{\frac{d+4}{d+2}} T^{\frac{4-d}{d+2}}.
\end{equation}

We describe our results regarding $Z_T$ in the following theorem.

\begin{theorem}  \label{prop-z}  
Let $Z_T$ be defined as in \eqref{eq:Q-def}. Then there exists a constant $C_{\ref{prop-z}  
}>0$ not depending on $N,T$ and $\beta$ such that, 
\begin{itemize}
\item[ \textbf{(i)}] for $d=1$, for all $\beta, T \geq 1$, 
\begin{equation*}
\log Z_T \geq -C_{\ref{prop-z}} \beta^{\frac{2}{3}} N^{\frac{5}{3}} T^{ },
\end{equation*}
\item[ \textbf{(ii)}] for $d=2,3$, for all $\beta \geq 1$ and $ 1\leq T \leq N$, 
\begin{equation*}
\log Z_T \geq- C_{\ref{prop-z}} \beta^{\frac{1}{2}} N^{\frac{3}{2}} T^{\frac{1}{2}} \log(\beta T),
\end{equation*}
\end{itemize} 
\end{theorem} 
The proof of Theorem \ref{prop-z} is given in Section \ref{sec-results}.

\section{Proofs of Theorems \ref{thm-radius} and \ref{prop-z} } \label{sec-results} 
 \begin{proof} [Proof of Theorem \ref{prop-z}]
In order to derive the bounds in Theorem \ref{prop-z}, we use a change of 
measure that will impose a time-dependent radial drift 
$(\lambda_k(t))_{t\ge0}$ on each Brownian motion 
$(B^{(k)}_t)_{t\ge0}$, of magnitude
\begin{equation} \label{drift-def}
|\lambda_k(t)| = \kappa(\alpha+1)t^\alpha,
\end{equation}
for some $\kappa>0$ and $\alpha<0$ to be determined.  Now we describe the 
directions $(\theta_k)_{k\in\{1,\dots,N\}}$ of the drifts, where each 
$\theta_k$ is a point on the unit sphere $\mathbf{S}_d\subset\mathbb{R}^d$.  
Let $\theta:=(\theta_k)_{k\in\{1,\dots,N\}}$ be an ensemble of i.i.d. 
random points chosen according to the uniform probability measure on 
$\mathbf{S}_d$.
Specifically, given $T,N$, we define $(\theta_k)_{k\in\{1,\dots,N\}}$ over a 
probability space $(\Omega_\theta,\mathcal{F}_\theta,P_\theta)$ and form the 
product space
\begin{equation*}
(\ol{\Omega},\mathcal{G},\mathbb{P})
= (\Omega_\theta\times\Omega,\mathcal{F}_\theta\times\mathcal{F}
        ,P_\theta\times P_T).
\end{equation*}
Sometimes we write $\mathbb{P}_{T}$ to emphasize the dependence on $T$.
Note that there is also an implicit dependence on $N$.
Finally, we denote $\lambda:=(\lambda_k)_{k\in\{1,\dots,N\}}$.

%\red{Specifically, we expand the probability space $\Omega$ defined after
%\eqref{eq:Bish-result} to a product space $\ol \Omega $ such for any 
%$(\omega_1,\omega_2) \in \ol \Omega$ such that $\omega_1$ is governed by 
%$P_T$ and $\omega_2$ is governed by $P_\theta$ is the uniform probability 
%measure on $\mathbf{S}_d$.}

Roughly speaking, we want the drift at time $t$ to be stronger than the 
standard deviation $t^{1/2}$ of Brownian motion at time $t$.  From 
\eqref{drift-def} it follows that the accumulated drift up to time $t$ is 
given by  
\begin{equation} \label{eq:drift}
\kappa (\alpha+1) \int_{0}^{t} s^{\alpha}ds =\kappa t^{\alpha+1}, 
\end{equation}
so we require that $\alpha>-1/2$.  

Given $\theta$ and $N,T$, we denote the measure 
on canonical path space $(C[0,T])^N$, which is induced by the drifts 
$\lambda$, as $P^\lambda_T$.  The Radon-Nikodym derivative with respect to 
$P_T$ is given by
\begin{equation*}
\frac{dP^\lambda_T}{d P_T  } 
 = \exp\left(\sum_{k=1}^{N}\left\{\int_{0}^{T}\lambda_k(t)\cdot dB^{(k)}_t
         -\frac{1}{2}\int_{0}^{T}|\lambda_k(t)|^2dt\right\}\right).
\end{equation*}
We can also define the corresponding product probability 
$\mathbb{P}^\lambda_T$ as before, using $P^\lambda_T$ instead of $P_T$.

We can express $Z_T$ in \eqref{eq:Q-def} in terms of $\mathcal{E}_T$ and the Radon-Nikodym 
derivative as follows
\begin{equation*}
Z_T = E^{\mathbb{P}_T^\lambda}\left[\mathcal{E}_T
         \left(\frac{dP^\lambda_T}{dP_T}\right)^{-1}\right].
\end{equation*}
We can use Jensen's inequality on $\log Z_T$, since the logarithm function 
is concave, to get   
\begin{equation} \label{z-split} 
\begin{split}
\log Z_T 
&\ge E^{\mathbb{P}_T^\lambda}\left[\log\mathcal{E}_T
           -\log\left(\frac{dP^\lambda_T}{dP_T} \right)\right] \\
&\ge -\beta E^{\mathbb{P}_T^\lambda}\left[\int_{\mathbb{R}^d}L_T(y)^2dy\right]
      -E^{\mathbb{P}_T^\lambda}\left[\log\left(\frac{dP^\lambda_T}{dP_T} \right)\right] \\
&=: -\beta I_{1}(d,\beta, N,T)- I_{2}(d,\beta, N, T).
\end{split}
\end{equation}

Using \eqref{drift-def} we can easily compute 
\begin{equation} \label{eq:log-RN}
\begin{split}
 I_2(d,\beta, N, T)
&= E^{\mathbb{P}_T^\lambda}\left[\sum_{k=1}^{N}\left\{\int_{0}^{T}\lambda_k(t)\cdot dB^{(k)}_t
         -\frac{1}{2}\int_{0}^{T}|\lambda_k(t)|^2dt\right\}\right]  \\
&= -\frac{N}{2}\kappa^2(\alpha+1)^2E\left[\int_{0}^{T}t^{2\alpha}dt\right]  \\
&= -\frac{\kappa^2(\alpha+1)^2}{2}N\cdot\frac{T^{2\alpha+1}}{2\alpha+1}.
\end{split}
\end{equation}
 
We can compare $I_2(d,\beta, N, T)$ in \eqref{eq:log-RN} with \eqref{z-asymp} in order to determine the constants in the drift. 
Neglecting multiplicative constants it follows that we must have 
\begin{equation*}
\kappa^2N\cdot\frac{T^{2\alpha+1}}{2\alpha+1}= \beta^{\frac{2}{d+2}} N^{\frac{d+4}{d+2}} T^{\frac{4-d}{d+2}}. 
\end{equation*}
This will lead to the following choice of drift parameters:
\begin{equation}  \label{old-drift} 
\kappa = \beta^{\frac{1}{d+2}}  N^{\frac{1}{d+2}} , \quad \alpha= - \frac{d-1}{d+2}, \quad d=1,2,3. 
\end{equation} 

\begin{remark} \label{rem-drift} 
While it is tempting to use the drift parameters in \eqref{old-drift} for $d=1,2,3$, we 
find that in the case of $d=3$ it is suboptimal. This choice of drift for 
$d=3$ yields the following bound on $Z_T$.  
\begin{equation}  \label{old-z-3} 
\log Z_T \geq  -C_{\ref{prop-z}}\beta^{\frac{3}{5}} N^{\frac{8}{5}} T^{\frac{4}{5}} \log(\beta T). 
\end{equation} 
Since \eqref{old-z-3} is suboptimal we do not prove it, but the interested 
reader can verify it by modifying the arguments in Sections 
\ref{sec-local-2-3}, \ref{sec-cross-2-3}, and \ref{sec-self-2-3}.

It is better to use the same drift parameter for $d=3$ as for $d=2$, namely
$\kappa = \beta^{\frac{1}{4}}  N^{\frac{1}{4}}$, $\alpha =-1/4$. 
This choice gives the bound in Theorem \ref{prop-z}(ii), that is for $d=3$,
\begin{equation*}
\log Z_T \geq- C_{\ref{prop-z}} \beta^{\frac{1}{2}} N^{\frac{3}{2}} T^{\frac{1}{2}} \log(\beta T).
\end{equation*}
%which is clearly better than \eqref{old-z-3} for $\beta,N,T \geq 1$. The 
%reason for that is that the bound on $Z_T$ depends on a comparison 
%between the two components on the right hand side of  \eqref{z-split}: 
%$I_{1}(d,\beta, N,T)$ which involves the occupation measure $L_T$, and on 
%$I_{2}(d,\beta, N,T)$ which contains the Radon-Nikodym derivative and is 
%computed in \eqref{eq:log-RN}. The bound on $I_{1}(d,\beta, N,T)$ in Proposition 
%\ref{prop-i-2} below, is obtained by bounding the terms in right-hand side 
%of \eqref{eq:L-squared}: $J_{1}(d,\beta,N,T)$ which depends on self-intersection 
%occupation measure of each branch of the polymer and $J_{2}(d,\beta,N,T)$ which 
%involves cross-intersection occupation measure of pairs of branches. The 
%bound on $J_{1}(d,\beta,N,T)$, which is derived on Section \ref{sec-self-2-3}, 
%is the more restrictive one and eventually determines the result of 
%Proposition \ref{prop-i-2}. In particular in the proof of Lemma 
%\ref{lem-j1-1} (see \eqref{b-bnd1}) we show that this bound should be 
%similar to $d=2,3$. Hence choosing similar drifts for the cases of $d=2,3$ 
%leads to an optimal bound for our method.  
\end{remark} 

For the reminder of the paper, we therefore fix 
  \begin{equation}\label{kap-al23} 
\kappa = \begin{cases} 
  \beta^{\frac{1}{3}}  N^{\frac{1}{3}}, & d=1, \\
  \beta^{\frac{1}{4}}  N^{\frac{1}{4}}, & d=2,3,\\
  \end{cases} 
   \qquad  \alpha  = \begin{cases} 0, & \  d=1, \\
   - \frac{1}{4}, &\  d=2,3. 
   \end{cases} 
\end{equation}

The rest of this paper is dedicated to the estimation of  $I_1(d,\beta, N,T)$. 
This is essentially given in the following proposition which together with 
\eqref{z-split} and \eqref{eq:log-RN} concludes the proof of Theorem 
\ref{prop-z}. 
\begin{prop} \label{prop-i-2} 
There exists a constant $C_{\ref{prop-i-2} }>0$ such that 

not depending on $N,T$ and $\beta$ such that, 
\begin{itemize}
\item[ \textbf{(i)}] for $d=1$, for all $N, \beta, T\geq 1$,  
\begin{equation*}
I_{1}(1,\beta, N,T) \leq C_{\ref{prop-i-2} }\beta^{-\frac{1}{3}} N^{\frac{5}{3}} T.
\end{equation*}
\item[ \textbf{(ii)}] for $d=2,3$, for all $\beta \geq 1$ and $ 2\leq T \leq N$, 
\begin{equation*}
I_{1}(d,\beta, N,T) \leq C_{\ref{prop-i-2} }\beta^{-\frac{1}{2}} N^{\frac{3}{2}} T^{\frac{1}{2}} \log(\beta T).
\end{equation*}
\end{itemize} 
\end{prop} 
The proof of Proposition \ref{prop-i-2} for $d=2,3$ is postponed to Section 
\ref{sec-local-2-3} and the case of $d=1$, which is much simpler, is 
postponed to Section \ref{sec-local-1}. 
\end{proof} 
 
\begin{proof}[Proof of Theorem \ref{thm-radius}] 
Note that Theorem \ref{prop-z} provides a lower bound on $Z_T$ of the form 
\begin{equation} \label{z-l-bnd-f} 
\log Z_T \geq - f(d,\beta,N,T), 
\end{equation} 
where
\begin{equation*}
f(d,\beta,N,T)=
\begin{cases}
C_{2.2}\beta^{2/3}N^{5/3}T & \text{if $d=1$},  \\
C_{2.2}\beta^{1/2}N^{3/2}T^{1/2} & \text{if $d=2,3$}.
\end{cases}
\end{equation*}

From \eqref{eq:Q-def}, \eqref{eq:def-A}, \eqref{p-def}, \eqref{eq:est-p-less} and \eqref{z-l-bnd-f} we have 
\begin{equation} 
\begin{aligned} 
\log Q_T\big(A^{(<)}_{T,r_1(T)}\big) &= \log  q_T^{(<)}  - \log Z_T \\
&\leq  -\beta C\frac{N^2T^2}{ (1+r_1(T))^d} + f(d,\beta,N,T), 
\end{aligned} 
\end{equation} 
hence by choosing $r_1(T)$ as in the upper bound in the statement of Theorem 
\ref{thm-radius}, with $c_d$ small enough and  $\beta,N,T$ as in the 
hypothesis, we ensure that $r_1(T) \geq 1$ and therefore, 
\begin{align*} 
 -\beta C\frac{N^2T^2}{(1+r_1(T))^d}  +f(d,\beta,N,T)  
 &\leq  -\beta \hat  C\frac{N^2T^2}{r_1(T)^d} +f(d,\beta,N,T) \\
 &< - c f(d,\beta,N,T), 
\end{align*} 
for some constant $c>0$ and we get the lower bound on $R_T$.  

Similarly from \eqref{eq:Q-def}, \eqref{eq:def-A}, \eqref{p-def}, \eqref{eq:est-p-greater} and \eqref{z-l-bnd-f} we have 
\begin{equation} 
\begin{aligned} 
\log Q_T\big(A^{(>)}_{T,r_2(T)}\big) &= \log  q_T^{(>)}  - \log Z_T \\
&\leq  -\frac{CNr_2(T)^2}{T} + f(d,\beta,N,T), 
\end{aligned} 
\end{equation} 
hence by choosing $r_2(T)$ as in the statement of Theorem \ref{thm-radius} we ensure that   
\begin{equation*}
-\frac{CNr_2(T)^2}{T}   +f(d,\beta,N,T) < - c' f(d,\beta,N,T), 
\end{equation*}
for some constant $c'>0$ and we get the upper bound on $R_T$. Note for $d=1$ 
that $r_1(T)$ and $r_2(T)$ agree up to a constant;  for $d=2$,  $r_1(T)$ and 
$r_2(T)$ agree up to logarithmic terms;  while in $d=3$ there is a gap 
between them. 
\end{proof} 

\section{Proof of Proposition \ref{prop-i-2} for $d=2,3$} \label{sec-local-2-3} 
This section is dedicated to the bound on $I_1(d,\beta, N,T)$ in  Proposition \ref{prop-i-2} for $d=2,3$.

\begin{proof}[Proof of Proposition \ref{prop-i-2} for $d=2,3$] 
 Recall that $L_T$ was defined in \eqref{eq:l-time}. From \eqref{z-split} 
and the statement of Proposition \ref{prop-i-2}, we see that we need the 
bound for $d=2,3$, 
\begin{equation} \label{l-sq-bnd} 
E^{\mathbb{P}_T^\lambda}\left[\int_{\mathbb{R}^d}L_T(y)^2dy\right] \leq \beta^{-\frac{1}{2}} N^{\frac{3}{2}} T^{\frac{1}{2}}. 
\end{equation} 
for $\alpha$ and $\kappa$ as in \eqref{kap-al23}. 

Note that 
\begin{equation*}
\begin{split}
\int_{\mathbb{R}^d}L_T(y)^2dy
&= \int_{\mathbb{R}^d}\left(\sum_{k=1}^{N}\int_{0}^{T}
     \mathbf{1}_{\mathbf{B}_1(y)}(B_t^{(k)})dt\right)^2dy  \\
&= \sum_{k,\ell=1}^{N}\int_{0}^{T}\int_{0}^{T}\left(\int_{\mathbb{R}^d}
    \mathbf{1}_{\mathbf{B}_1(y)}(B_t^{(k)})
       \mathbf{1}_{\mathbf{B}_1(y)}(B_s^{(\ell)}) dy\right)dtds.  
\end{split}
\end{equation*}
In fact, for $a,b\in\mathbb{R}^d$, we have
\begin{equation*}
\int_{\mathbb{R}^d}\mathbf{1}_{\mathbf{B}_1(y)}(a)\mathbf{1}_{\mathbf{B}_1(y)}(b)dy
\le \mathbf{1}_{\mathbf{B}_2(0)}(b-a)|\mathbf{B}_1(0)|
= C_d \mathbf{1}_{\mathbf{B}_2(0)}(b-a).
\end{equation*}

Thus, if $f^{(k,\ell,\alpha)}_{t,s}(z)$ is the probability density function of 
$B^{(k)}_t-B^{(\ell)}_s$ under $\mathbb{P}_T^\lambda$, then using Fubini's Theorem we get, 
\begin{equation} \label{eq:L-squared}
\begin{aligned} 
\int_{\mathbb{R}^d} E^{\mathbb{P}_T^\lambda}\left[L_T(y)^2\right]dy 
&\le C_d \sum_{k,\ell=1}^{N} \iint_{[0,T]^2}\int_{\mathbf{B}_2(0)} f^{(k,\ell,\alpha)}_{t,s}(z) dz dsdt \\ 
& = C_d \sum_{k=1}^{N} \iint_{[0,T]^2}\int_{\mathbf{B}_2(0)} f^{(k,k,\alpha)}_{t,s}(z) dz dsdt \\ 
& \quad + C_d \sum_{k \not =\ell} \iint_{[0,T]^2}\int_{\mathbf{B}_2(0)} f^{(k,\ell,\alpha)}_{t,s}(z) dz dsdt \\ 
& =:J_{1}(d,\beta,N,T)+J_{2}(d,\beta,N,T). 
\end{aligned} 
\end{equation}
Note that $J_1(d,\beta,N,T)$ represents the sum of mean-squared 
self-intersection occupation measure of each branch of the star polymer. 
$J_2(d,\beta,N,T)$ represents the sum over all pairs of branches of their 
mean-squared cross-intersection occupation measure. 
The following proposition gives some essential bounds on each of these terms. 
\begin{prop} \label{prop-j} 
For $d=2,3$ there exists a constant $C>0$ 
such that for all $1\leq T \leq N$, $N\in\mathbb{Z}^+$ and $\beta  \geq 1$ the following bound holds:
\begin{equation*}
J_{i}(d,\beta, N,T) \leq C \beta^{-\frac{1}{2}} N^{\frac{3}{2}} T^{\frac{1}{2}} \log(\beta T), \quad   i=1,2. 
\end{equation*}
\end{prop} 
The proof of Proposition \ref{prop-j} is long and involved.  In Section 
\ref{sec-cross-2-3} we prove the bound on cross-intersection occupation measure, which involves two different Brownian motions $B^{(k_i)}$ for $i=1,2$ and $k_1\ne k_2$,
. The bound on self-intersection occupation measure 
is derived in Section \ref{sec-self-2-3}. From \eqref{l-sq-bnd}, \eqref{eq:L-squared} and 
Proposition \ref{prop-j} we get the bounds in Proposition \ref{prop-i-2}. 
\end{proof} 

\section{Cross-intersections occupation measure for $d=2,3$} \label{sec-cross-2-3}
In this section we derive the bounds on $J_{2}(d,\beta, N,T)$ from Proposition \ref{prop-j} for $d=2,3$. 
We recall that $f^{(k,\ell,\alpha)}_{t,s}(z)$ is the probability density function of 
$B^{(k)}_t-B^{(\ell)}_s$ under $\mathbb{P}_T^\lambda$. In the following we fix $k$ and $\ell$ and let $\theta$ be the angle between the drifts of $B^{(k)}$ and $B^{(\ell)}$, which were sampled according to the procedure described after \eqref{drift-def}. 

Our next task is to estimate $f^{(k,\ell,\alpha)}_{t,s}(z)$.  Since $f$ only 
depends on the angle $\theta$, from now on we write 
$f^{\theta,\alpha}_{t,s}$ for our probability density.  

Instead of working with $P^\lambda_T$, we will work with $P_T$ and consider 
processes $B^{(k)}_t+D^{(k)}_t,B^{(\ell)}_s+D^{\ell}_s$, where 
$D^{(k)},D^{(\ell)}$ are the respective drift processes. First note that
\begin{equation*}
B^{(k)}_t-B^{(\ell)}_s \sim \mathcal N(0,t+s),
\end{equation*}
where $\mathcal N(\cdot,\cdot)$ stands for the $d$-dimensional normal distribution.

Recall that the drift magnitudes are given by \eqref{eq:drift},
namely
\begin{align} \label{drift} 
|D^{(k)}_t|=\kappa t^{\alpha+1},  && |D^{(\ell)}_s|=\kappa s^{\alpha+1}.
\end{align}

For the remainder of the proof we assume that $\theta\in[0,\pi]$.  
First we get a lower bound on $|z-D^{(k)}_t+D^{(\ell)}_s|^2$.
Let $a=D^{(k)}_t-D^{(\ell)}_s$. Then
\begin{equation*}
\begin{split}
|a|^2 = |(a-z)+z|^2 \le 2|z-a|^2+2|z|^2 \le 2|z-a|^2+8,
\end{split}
\end{equation*}
since $|z|\le2$ in the domain of integration of $J_2(d,\beta, N,T)$ (see \eqref{eq:L-squared}). 
Now we use the law of cosines to deduce
\begin{equation} \label{d-r-1} 
\begin{split}
|D^{(k)}_t-D^{(\ell)}_s|^2 
&=|D^{(k)}_t|^2+|D^{(\ell)}_s|^2-2|D^{(k)}_t|\cdot|D^{(\ell)}_s|\cos\theta \\
&=\left(|D^{(k)}_t|-|D^{(\ell)}_s|\right)^2
       +2|D^{(k)}_t|\cdot|D^{(\ell)}_s|\left[1-\cos\theta\right].
\end{split}
\end{equation}
Elementary geometry shows that for some constant $C>0$ and for all
$\theta \in [0,\pi]$, 
\begin{equation} \label{cosine-th} 
1-\cos\theta \ge C\theta^2.
\end{equation}
In fact we could choose $C=2/\pi^2$.  Therefore,
\begin{equation} \label{eq:law-cosines}
|D^{(k)}_t-D^{(\ell)}_s|^2 
\ge \left(|D^{(k)}_t|-|D^{(\ell)}_s|\right)^2
       +C|D^{(k)}_t|\cdot|D^{(\ell)}_s|\theta^2.
\end{equation}

We also recall the following fact.  
For any $K>0$ we have
\begin{equation} \label{eq:est-int-exponential}
\int_{0}^{\infty}\exp\left(-K\theta^2\right)d\theta
= \frac{\sqrt{\pi}}{2\sqrt{K}}. 
\end{equation} 
%\begin{remark} 
%The bound in \eqref{eq:est-int-exponential} will be used in Lemma \ref{lem-j-r1} (see \eqref{id-use}) and in Lemma 
%\ref{lem-j-r3} (see \eqref{eq:bound-int-theta}) in order to bound 
%$J_2(d,\beta,N,T)$ for two different cases. 
%Note that the aforementioned bounds can be improved for $d=3$ by using the 
%following equality for any $K>0$, 
%\begin{equation*}  
%\int_{0}^{\infty}\exp\left(-K\theta^2\right) \theta d\theta
%= \frac{1}{2K}, 
%\end{equation*}
%for the expression for $J_2(d,\beta,N,T)$ in \eqref{j-2-theta} which has an 
%extra factor of $\theta$ (see explanation after  \eqref{j-2-theta}). As mentioned in Remark \ref{rem-drift} this will 
%not have any affect on the result of Theorem \ref{prop-z}. Therefore, for 
%the sake of simplicity we provide one proof for the bounds $J_2(d,\beta, N,T)$ 
%from Proposition \ref{prop-j} for $d=2,3$ using 
%\eqref{eq:est-int-exponential}. 
%\end{remark}  

We derive a preliminary bound for $J_{2}(d,\beta,N,T)$ which appears in
\eqref{eq:L-squared}, for $d\in\{2,3\}$:
\begin{equation}\label{re0}
\begin{aligned} 
J_{2}(d,\beta,N,T) 
& \leq C N^2\int_{0}^{\pi}\iint_{0\le s\le t\le T}\frac{1}{(t+s)^{d/2}}\\
& \qquad \times \int_{\mathbf{B}_2(0)}
  \exp\left(-\frac{\big|z-D_t^{(k)}+D_s^{(\ell)}\big|^2}{2(t+s)}\right)
        dz dsdt d\theta.
        \end{aligned}
\end{equation}
%\textbf{Case 1: $d=2$.} 
%Note that $J_{2}(2,\beta,N,T)$ is bounded by 
%\begin{equation}\label{re0}
%\begin{aligned} 
%J_{2}(2,\beta,N,T) 
%& \leq C N^2\int_{0}^{\pi}\iint_{0\le s\le t\le T}\frac{1}{t+s}\\
%& \qquad \times \int_{\mathbf{B}_2(0)}
%  \exp\left(-\frac{\big|z-D_t^{(k)}+D_s^{(\ell)}\big|^2}{2(t+s)}\right)
%        dz dsdt d\theta.
%        \end{aligned}
%\end{equation}
%
%\textbf{Case 2: $d=3$.}  
%By the same argument as in Case 1 we get,  
%\begin{equation} \label{j-2-theta} 
%\begin{aligned} 
%J_{2}(3,\beta,N,T) &\leq CN^2\int_{0}^{\pi}\iint_{0\le s\le t\le T}\frac{1}{(t+s)^{3/2}} \\
%& \qquad \times\int_{\mathbf{B}_2(0)}
%  \exp\left(-\frac{\big|z-D_t^{(k)}+D_s^{(\ell)}\big|^2}{2(t+s)}\right)
%        dz dsdt d\theta. 
%        \end{aligned}
%\end{equation}

Note that for $d=3$ in the integral above, using polar coordinates would 
give $\theta d\theta$.  The interested reader can check that using $d\theta$ 
gives the same end result.  Of course we can bound $\theta$ by $\pi$ and so 
replace $\theta d\theta$ by $d\theta$.

Let $V$ be the subspace generated by the first two coordinates of 
$\mathbb{R}^3$, so for $z=(z_1,z_2,z_3)$ we write $z_V = (z_1,z_2)$. In 
fact we can find a coordinate system in which $D_t^{(k)}$ is parallel to 
$z_1$, so the above integral is bounded by 
\begin{equation} \label{re1} 
\begin{aligned} 
&J_{2}(3,\beta,N,T)\\
 &\leq CN^2\int_{0}^{\pi}\iint_{0\le s\le t\le T}\frac{1}{(t+s)^{}}\int_{ {\mathbf{B}}^V_2(0)}
  \exp\left(-\frac{\big|  z_V-D_t^{(k)}+D_s^{(\ell)}\big|^2}{2(t+s)}\right)
        dz_V  \\
        &\qquad\qquad \qquad \quad \qquad \qquad \times \frac{1}{(t+s)^{1/2}}  \int_{-\infty}^{\infty}  \exp\left(-\frac{z_3^2}{2(t+s)}\right) dz_3 dsdt d\theta \\
        &\leq CN^2\int_{0}^{\pi}\iint_{0\le s\le t\le T}\frac{1}{(t+s)^{}} \\
        & \qquad \qquad   \quad \qquad \qquad \times \int_{ {\mathbf{B}}^V_2(0)}
  \exp\left(-\frac{\big|z_V-D_t^{(k)}+D_s^{(\ell)}\big|^2}{2(t+s)}\right)
        dz_V  dsdt d\theta,
\end{aligned} 
\end{equation}
where $ {\mathbf{B}}^V_2(0)$ is the projection of ${\mathbf{B}}_2(0)$ to $V$. 
From \eqref {re0} and \eqref {re1} it follows that we need to bound the same 
integral for the cases where $d=2,3$ but $D_t^{(k)},D_s^{(\ell)}$ will 
change according to the dimension as implied by \eqref{kap-al23} and 
\eqref{drift}. In the following we therefore work on to the combined cases 
$d=2,3$.  

 In order to do that we distinguish between a few cases which depend on the range of $(t,s)$ in the above integral. 

Define 
\begin{equation} \label{R-def} 
\begin{aligned} 
\mathbf{R}_1&=\left \{(s,t): 0\le s\le t\le T, t\ge4 \right\}, \\ 
\mathbf{R}_2&=\left\{(s,t): 0\le s\le t\le T, \, t < 4, \,   |D_t^{(k)}-D_s^{(\ell)}|^2  \leq 8\right \},\\ 
\mathbf{R}_3&=\left\{(s,t): 0\le s\le t\le T, \, t < 4, \, |D_t^{(k)}-D_s^{(\ell)}|^2 > 8 \right\}. 
\end{aligned} 
\end{equation}
Moreover for $i=1,2,3,$ and $d=2,3$ define: 
\begin{equation} \label{bar-J} 
\begin{aligned} 
\ol J (d,\beta,N,T,\mathbf{R}_i) &= C N^2\int_{0}^{\pi}\iint_{\mathbf{R}_i}\frac{1}{t+s}\\
&  \times \int_{\mathbf{B}_2(0)}
  \exp\left(-\frac{\big|z-D_t^{(k)}+D_s^{(\ell)}\big|^2}{2(t+s)}\right)
        dz dsdt d\theta.
\end{aligned} 
\end{equation}
We now present a sequence of technical lemmas that will help us to bound 
$J_2(d,\beta,N,T)$. The proof of Proposition \ref{prop-j} for $J_2(d,\beta, N,T)$ 
is given in the end of this section. 
\begin{lemma} \label{lem-j-r1} 
For $d=2,3$, there exists a constant 
$C_{\ref{lem-j-r1}}>0$ such that for all  $\beta, T \geq 1$ and $N \in \mathbb{Z}^{+}$ the following bound holds: 
\begin{equation*}
\ol J (d,\beta,N,T,\mathbf{R}_1)  \leq C_{\ref{lem-j-r1}} \beta^{-1/2} N^{3/2} T^{1/2}.
\end{equation*}
\end{lemma} 

\begin{proof} 
Using \eqref{eq:law-cosines} we get
\begin{equation*}
|z-D_t^{(k)}+D_s^{(\ell)}|^2 
\ge   \left(|D^{(k)}_t|-|D^{(\ell)}_s|\right)^2  
       +C|D^{(k)}_t|\cdot|D^{(\ell)}_s|\theta^2  - 4,
\end{equation*}
since $|z|\le2$ in the domain of integration of $\ol J (d,\beta,N,T,\mathbf{R}_i)$ (see \eqref{bar-J}). 

Note that since on $\mathbf{R}_1$ we have $s+t>4$, it follows that $4/(s+t)\le1$ and by integrating over $z$ we get  
\begin{equation*}
\begin{split}
&\ol J (d,\beta,N,T,\mathbf{R}_1) \\
&=N^2\int_{0}^{\pi}\iint_{(s,t)\in\mathbf{R}_1}\frac{1}{t+s}\int_{\mathbf{B}_2(0)}
  \exp\left(-\frac{\big|z-D_t^{(k)}+D_s^{(\ell)}\big|^2}{2(t+s)}\right)
        dz dsdt d\theta  \\
&\le CN^2\int_{0}^{\pi}\iint_{(s,t)\in\mathbf{R}_1}\frac{1}{t+s}  \\
&\hspace{1cm}  \times\exp\left(-\frac{\left(|D^{(k)}_t|-|D^{(\ell)}_s|\right)^2
       +C|D^{(k)}_t|\cdot|D^{(\ell)}_s|\theta^2}
       {2(t+s)} + 1\right) dsdt d\theta.
\end{split}
\end{equation*}
Next we factor the above exponential, incorporate $e^1$ into the constant, 
and apply our elementary integral \eqref{eq:est-int-exponential} to obtain
\begin{equation} \label{id-use} 
\begin{split}
&\ol J (d,\beta,N,T,\mathbf{R}_1)\\
&\le 
CN^2\iint_{(s,t)\in\mathbf{R}_1}\frac{1}{t+s}
\exp\left(-\frac{\left(|D^{(k)}_t|-|D^{(\ell)}_s|\right)^2}{2(t+s)}\right) \\
&\qquad\qquad \qquad\qquad\qquad\qquad \times \left[\int_{0}^{\pi}
    \exp\left(-\frac{C|D^{(k)}_t|\cdot|D^{(\ell)}_s|\theta^2}{2(t+s)}\right)
      d\theta\right]  dsdt  \\
&=
CN^2\iint_{(s,t)\in\mathbf{R}_1}\frac{1}{\sqrt{(t+s)|D^{(k)}_t|\cdot|D^{(\ell)}_s|}} \\
&\qquad\qquad \qquad\qquad\qquad\qquad \times \exp\left(-\frac{\left(|D^{(k)}_t|-|D^{(\ell)}_s|\right)^2}{2(t+s)}\right) dsdt.
\end{split}
\end{equation}
Then we use the fact that $0\le s\le t$, and so $|D_s^{(\ell)}|\le|D_t^{(k)}|$. From \eqref{kap-al23} and \eqref{drift} we have 
$|D^{(k)}_t|=\beta^{1/4}N^{1/4}t^{3/4}$, so we get   

\begin{equation*}
\begin{split}
&\ol J (d,\beta,N,T,\mathbf{R}_1)\\
&\le \tilde CN^2\iint_{(s,t)\in\mathbf{R}_1}t^{-\frac{1}{2}}\beta^{-\frac{1}{4}}N^{-\frac{1}{4}}
 t^{-\frac{3}{8}}s^{-\frac{3}{8}}  \\
 &\qquad \qquad  \qquad \qquad \times \exp\left(-C\frac{\beta^{\frac{1}{2}}N^{\frac{1}{2}}[t^{\frac{3}{4}}-s^{\frac{3}{4}}]^2}{t}\right) dsdt\\
&\le \tilde C\beta^{-\frac{1}{4}}N^{\frac{7}{4}}\int_{0}^{T}\int_{0}^{t}t^{-\frac{7}{8}}s^{-\frac{3}{8}} \\
 &\qquad \qquad  \qquad \qquad \times
    \exp\left(-C\beta^{\frac{1}{2}}N^{\frac{1}{2}}t^{-1}[t^{\frac{3}{4}}-s^{\frac{3}{4}}]^2\right) dsdt. 
\end{split}
\end{equation*}

Now by making a change variables to $u=s/T$ and $v=t/T$ it follows that,  
\begin{equation}\label{eq-rd} 
\begin{split}
\ol J (d,\beta,N,T, & \mathbf{R}_1)
  =\tilde C\beta^{-\frac{1}{4}}N^{\frac{7}{4}}\int_{0}^{1}\int_{0}^{v}(Tv)^{-\frac{7}{8}}(Tu)^{-\frac{3}{8}} \\
& \times\exp\left(-C\beta^{\frac{1}{2}}N^{\frac{1}{2}}(Tv)^{-1}
           [(Tv)^{\frac{3}{4}}-(Tu)^{\frac{3}{4}}]^2\right) d(Tu)d(Tv)  \\
&\hspace{-2cm} 
  = \tilde C\beta^{-\frac{1}{4}}N^{\frac{7}{4}}T^{\frac{3}{4}}\int_{0}^{1}\int_{0}^{v}v^{-\frac{7}{8}}u^{-\frac{3}{8}} 
\exp\left(-C\beta^{\frac{1}{4}}N^{\frac{1}{4}}T^{\frac{1}{ 2}}v^{-1}
           \big[v^{\frac{3}{4}}-u^{\frac{3}{4}}\big]^2\right) dudv. 
\end{split}
\end{equation}
We will show that the following bound holds for all $K\geq 1$,  
\begin{equation} \label{eq:need-integral}
\begin{split} 
\int_{0}^{1}\int_{0}^{v}v^{-\frac{7}{8}}u^{-\frac{3}{8}}  \exp\left(-CK v^{-1}
           [v^{\frac{3}{4}}-u^{\frac{3}{4}}]^2\right) dudv 
\leq K^{-1/2}. 
\end{split}
\end{equation}
We choose $K=\beta^{1/2}N^{1/2}T^{1/2}$ as in \eqref{eq-rd} so the bound in \eqref{eq:need-integral} will give us  
\begin{equation*}
\ol J (d,\beta,N,T,\mathbf{R}_1) \leq C\beta^{-1}N^{}, 
\end{equation*}
for $\beta, T \geq 1$, and this will complete the proof for $d=2$ and $d=3$. 

 By the mean value theorem, there exists $r\in(u,v)$ such that
\begin{equation*}
v^{\frac{3}{4}}-u^{\frac{3}{4}} = cr^{-\frac{1}{4}}(v-u) \ge cv^{-\frac{1}{4}}(v-u).
\end{equation*}
Hence following \eqref{eq-rd}, it is enough to bound
\begin{equation} \label{int-t}
\begin{split}
&\int_{0}^{1}\int_{0}^{v}v^{-\frac{7}{8}}u^{-\frac{3}{8}}\exp\left(-Kv^{-\frac{3}{2}}[v-u]^2\right) dudv \\
&= \int_{0}^{1}\int_{0}^{v}v^{-\frac{7}{8}}(v-w)^{-\frac{3}{8}}\exp\left(-Kv^{-\frac{3}{2}}w^2\right) dwdv.
\end{split}
\end{equation}
We examine the inner integral in \eqref{int-t}, and write
\begin{equation} \label{k-comp-1} 
\begin{split}
&\int_{0}^{v}v^{-\frac{7}{8}}(v-w)^{-\frac{3}{8}}\exp\left(-Kv^{-\frac{3}{2}}w^2\right) dw \\
&= \int_{0}^{v/2}v^{-\frac{7}{8}}(v-w)^{-\frac{3}{8}}\exp\left(-Kv^{-\frac{3}{2}}w^2\right) dw \\
&\quad  + \int_{v/2}^{v}v^{-\frac{7}{8}}(v-w)^{-\frac{3}{8}}\exp\left(-Kv^{-\frac{3}{2}}w^2\right) dw
  \\
&=: H_1(v)+ H_2(v).
\end{split}
\end{equation}

First dealing $H_1$, we have
\begin{equation*}
\begin{split}
H_1(v) &= \int_{0}^{v/2} (v-w)^{-\frac{3}{8}}\exp\left(-Kv^{-\frac{3}{2}}w^2\right)dw  \\
&\le Cv^{-\frac{3}{8}}\int_{0}^{\infty}\exp\left(-Kv^{-\frac{3}{2}}w^2\right)dw \\
&\le Cv^{-\frac{3}{8}}K^{-1/2}v^{\frac{3}{4}}  \\
&\le Cv^{\frac{3 }{8}}K^{-1/2},
\end{split}
\end{equation*}
using \eqref{eq:est-int-exponential}. Integrating over $v$ as well, we get  
\begin{equation} \label{k-comp0} 
\begin{split}
\int_{0}^{1}v^{-\frac{7}{8}}H_1(v)dv &= CK^{-1/2}\int_{0}^{1}v^{-\frac{7}{8}}v^{\frac{3 }{8}}dv \\
&= CK^{-1/2}.
\end{split}
\end{equation}

Turning to $H_2$, we have 
\begin{equation*}
\begin{split}
H_2(v) &= \int_{v/2}^{v}(v-w)^{-\frac{3}{8}}\exp\left(-Kv^{-\frac{3}{2}}w^2\right)dw  \\
&\le \exp\left(-CKv^{-\frac{3}{2}}v^2\right)\int_{v/2}^{v}(v-w)^{-\frac{3}{8}}dw  \\
&\le Cv^{\frac{5}{8}}\exp\left(-CKv^{\frac{1}{2}}\right).
\end{split}
\end{equation*}
Integrating over $v$ as well, we get
\begin{equation} \label{k-comp} 
\begin{split}
\int_{0}^{1}v^{-\frac{7}{8}}H_2(v)dv 
&\le \int_{0}^{1}v^{-\frac{7}{8}}v^{\frac{5}{8}}\exp\left(-CKv^{\frac{1}{2}}\right)dv  \\
&= \int_{0}^{1}v^{-\frac{1}{4}} \exp\left(-CKv^{\frac{1}{2}}\right)dv  \\
&= K^{-\frac{3}{2}}\int_{0}^{1} (K^{2}v)^{-\frac{1}{4}}  \\
& \qquad \qquad\quad \times\exp\left(-C(K^{2}v)^{\frac{1}{2}}\right)d(K^{2}v)  \\
&= K^{ -\frac{3}{2}}\int_{0}^{1} r^{-\frac{1}{4} }\exp\left(-Cr^{\frac{1}{2}}\right)dr  \\
&\leq K^{-3/2}, \quad \textrm{for all } K\geq 1. 
\end{split}
\end{equation}
Note that we have used 
\begin{equation*}
\int_0^1 x^{-\eta/2} e^{-x^\eta} dx < \infty \quad \textrm{for all } \eta \in (0,2). 
\end{equation*}
From \eqref{int-t}, \eqref{k-comp-1}, \eqref{k-comp0} and \eqref{k-comp} we 
get \eqref{eq:need-integral} and Lemma \ref{lem-j-r1} follows. 
 \end{proof} 

The following technical lemma gives us some essential bounds on 
$\ol J (d,\beta,N,T,\mathbf{R}_2)$ which was defined in \eqref{bar-J}.  

\begin{lemma} \label{lem-j-r2} For $d=2,3$, there exists a constant $C_{\ref{lem-j-r2} }>0$ such that for all $\beta  \geq 1, \, T> 0, \, N\in\mathbb{Z}^+$ the 
following bound holds: 
\begin{equation*}
\ol J (d,\beta,N,T,\mathbf{R}_2)  \leq C_{\ref{lem-j-r2} } \beta^{-1/2} N^{3/2}.  
\end{equation*}
\end{lemma} 

\begin{proof} 
Note that on $\mathbf{R}_2$ since $|z|^2 <4$ we we expect 
$|z-D_t^{(0)}-D_s^{(\ell)}|^2$ to be of the same order of $s+t$. In this 
case we get that the integral over $z$ in the right-hand side of 
\eqref{bar-J} is approximately 1, that is,
\begin{equation*}
\int_{\mathbf{B}_2(0)}
 \frac{1}{s+t}\exp\left(-C\frac{|z-D_t^{(k)}+D_s^{(\ell)}|^2}{2(s+t)}\right) dz
\approx 1.
\end{equation*}
Hence we will absorb this integral as a multiplicative factor in our bounds on $\ol J (d,\beta,N,T,\mathbf{R}_2)$ as follows, 
\begin{equation}  \label{nn-1} 
\begin{aligned} 
\ol J (d,\beta,N,T,\mathbf{R}_2) &= C N^2\int_{0}^{\pi}\iint_{\mathbf{R}_2}\frac{1}{t+s}\\
&\qquad\qquad  \times \int_{\mathbf{B}_2(0)}
  \exp\left(-\frac{\big|z-D_t^{(k)}+D_s^{(\ell)}\big|^2}{2(t+s)}\right) 
        dz dsdt d\theta \\
        &\leq  C N^2\int_{0}^{\pi}\iint_{\mathbf{R}_2} dsdt d\theta. 
\end{aligned} 
\end{equation}
 
Now we flesh out this heuristic discussion.  
From \eqref{kap-al23} and \eqref{drift} we get 
\begin{equation} \label{eq:diff-D}
|D_t^{(k)}-D_s^{(\ell)}|^2
= \beta^{\frac{1}{2} }N^{\frac{1}{2}}\left(t^{\frac{3}{4}}-s^{\frac{3}{4}}\cos\theta\right)^2
+ \beta^{\frac{1}{2}}N^{\frac{1}{2}}\left(s^{\frac{3}{4}}\sin\theta\right)^2.
\end{equation}
Suppose we are on the region $\mathbf{R}_2$.  
Then $|D_t^{(k)}-D_s^{(\ell)}|^2 \leq 8$, and the terms on the 
right side of \eqref{eq:diff-D}, which are both positive, must each be
bounded by 8.  Therefore
$\beta^{1/2}N^{1/2}\left(s^{3/4}\sin\theta\right)^2\le 8$ and so using 
$\sin\theta \geq C\theta$ for $\theta \in [0 ,\pi]$ we have  
\begin{equation} \label{const1} 
\theta \le C s^{-\frac{3}{4}}\beta^{-\frac{1}{4}}N^{-\frac{1}{4}}.
\end{equation}
From \eqref{eq:diff-D} we also have  
\begin{equation} \label{eq:less8}
\beta^{\frac{1}{2}}N^{\frac{1}{2}}\left(t^{\frac{3}{4}}-s^{\frac{3}{4}}\cos\theta\right)^2 \le 8.
\end{equation}
Solving \eqref{eq:less8} for $s$, we get
\begin{equation*}
s\ge \frac{1}{\cos\theta}\left(t^{3/4}-8^{1/2}\beta^{-1/4}N^{-1/4}\right)^{4/3}
     \ge \left(t^{3/4}-8^{1/2}\beta^{-1/4}N^{-1/4}\right)^{4/3}
\end{equation*}
Therefore if $t^{3/4}>8^{1/2}\beta^{-\frac{1}{4}}N^{-\frac{1}{4}}$, or 
equivalently $t>4\beta^{-\frac{1}{3}}N^{-\frac{1}{3}}$ we get 
\begin{equation} \label{y1} 
0<\left(t^{\frac{3}{4}}-8^{\frac{1}{2}}N^{-\frac{1}{4}}\beta^{-\frac{1}{4}}\right)^{\frac{4}{3}}
< s \leq t.
\end{equation}
Together with \eqref{nn-1} and \eqref{const1} we get 
\begin{equation}  \label{j-d-one} 
\begin{aligned} 
&\ol J (d,\beta,N,T,\mathbf{R}_2) \\
& \leq  C N^2\int_{\theta\le C s^{-\frac{3}{4}}\beta^{-\frac{1}{4}}N^{-\frac{1}{4}}}\iint_{\mathbf{R}_2} dsdt d\theta\\ 
 &\leq C\beta^{-\frac{1}{4}}N^{\frac{7}{4}} \iint_{\mathbf{R}_2} s^{-\frac{3}{4}} dsdt   \\
  &\leq C\beta^{-\frac{1}{4}}N^{\frac{7}{4}} \int_{4\beta^{-\frac{1}{3}}N^{-\frac{1}{3}}}^{4} \int_{(t^{\frac{3}{4}}-8^{\frac{1}{2}}N^{-\frac{1}{4}}\beta^{-\frac{1}{4}})^{\frac{4}{3}}}^ts^{-\frac{3}{4}} dsdt   \\
&\quad  
+   C\beta^{-\frac{1}{4}}N^{\frac{7}{4}} \int_0^{4\beta^{-\frac{1}{3}}N^{-\frac{1}{3}}} \int_0^{4}s^{-\frac{3}{4}} dsdt   \\
  &\leq C\beta^{-\frac{1}{4}}N^{\frac{7}{4}} \int_{4\beta^{-\frac{1}{3}}N^{-\frac{1}{3}}}^{4}  \big[t^{\frac{1}{4}}- (t^{\frac{3}{4}}-8^{\frac{1}{2}}N^{-\frac{1}{4}}\beta^{-\frac{1}{4}})^{\frac{1}{3}}\big]dt   \\
  &\quad  +   C\beta^{-\frac{1}{4}}N^{\frac{7}{4}} \int_0^{4\beta^{-\frac{1}{3}}N^{-\frac{1}{3}}}dt   \\
&=:H_1(\beta,N) +H_2(\beta,N). 
\end{aligned} 
\end{equation}
Now since $\beta,N\ge1$, we have 
$\beta^{-1/3}N^{-1/3}\le\beta^{-1/4}N^{-1/4}$ and
\begin{equation} \label{h-1-bd}
H_1(\beta,N)\le C\beta^{-1/4}N^{7/4}\beta^{-1/4}N^{-1/4} = C\beta^{-1/2}N^{3/2}.
\end{equation}
Turning to $H_2(\beta,N)$, we use the fact that for every $0<x< y < \infty$, 
\begin{equation*}
\begin{split}
x^{1/3}-y^{1/3} &= \frac{(x^{1/3}-y^{1/3})(x^{2/3}+x^{1/3}y^{1/3}+y^{2/3})}
          {x^{2/3}+x^{1/3}y^{1/3}+y^{2/3}}    \\
&= \frac{x-y}{x^{2/3}+x^{1/3}y^{1/3}+y^{2/3}} \\
&\leq \frac{x-y}{x^{2/3}}.\\
\end{split}
\end{equation*}
By taking 
\begin{align*}
x=t^{3/4}  &&  y=t^{3/4}-8^{1/2}N^{-1/4}\beta^{-1/4}, 
\end{align*} 
and using \eqref{y1} we get the following bound, 
\begin{equation*}
\begin{split}
t^{\frac{1}{4}}- (t^{\frac{3}{4}}-8^{1/2}N^{-\frac{1}{4}}\beta^{-\frac{1}{4}})^{\frac{1}{3}}
&= \frac{8^{1/2}N^{-1/4}\beta^{-1/4}}{t^{1/2} }  \\
&\le CN^{-1/4}\beta^{-1/4}t^{-1/2}.
\end{split}
\end{equation*}
Together with \eqref{j-d-one} it follows that, 
\begin{equation} \label{h-2-bd}
\begin{split}
H_2(\beta,N) &\le C\beta^{-1/4}N^{7/4}\int_{0}^{4}\beta^{-1/4}N^{-1/4}t^{-1/2}dt  \\
&\le C\beta^{-1/2}N^{3/2}.
\end{split}
\end{equation}
By plugging in \eqref{h-1-bd} and \eqref{h-2-bd} into \eqref{j-d-one} we complete the proof.   

%where we have used, 
%\begin{equation*}
%t^{\frac{1}{2}}- (t^{\frac{3}{4}}-4N^{-\frac{1}{4}}\beta^{-\frac{1}{4}})^{\frac{1}{3}} \leq Ct^{-\frac{1}{2}} N^{-\frac{1}{4}}\beta^{-\frac{1}{4}}.
%\end{equation*}
This completes the proof of Lemma \ref{lem-j-r2}. 
\end{proof} 

The following lemma gives us some essential bounds on $\ol J_d (\beta,N,T,\mathbf{R}_3)$ which was defined in 
\eqref{bar-J}. 

 \begin{lemma} \label{lem-j-r3} 
For $d=2,3$ there exists a constant $C>0$ such that the following bound hold for  $\beta \geq 1$, $T>0$ and $N\in\mathbb{Z}^+$,
\begin{equation*}
\ol J (d,\beta,N,T,\mathbf{R}_3)  \leq C_{\ref{lem-j-r3} } \beta^{-1/2} N^{3/2}. 
\end{equation*}
 \end{lemma}

\begin{proof} 
Recalling that $|z|<2$, we have 
\begin{equation} \label{k12} 
|z-D_t^{(k)}+D_s^{(\ell)}|^2 \ge |D_t^{(k)}+D_s^{(\ell)}|^2 - 4^2 >  \frac{1}{2} |D_t^{(k)}+D_s^{(\ell)}|^2
\end{equation}
Using \eqref{kap-al23}, \eqref{drift} and \eqref{eq:law-cosines} we get 
\begin{equation} \label{eq:lower-bound-numerator}
\begin{split}
|D_t^{(k)}+D_s^{(\ell)}|^2
   & \ge \big(|D^{(k)}_t|-|D^{(\ell)}_s|\big)^2 + C|D^{(k)}_t|\cdot|D^{(\ell)}_s|\theta^2   \\
&\ge\beta^{1/2}N^{1/2}\left[\big(t^{3/4}-s^{3/4}\big)^2 
     + Ct^{3/4}s^{3/4}\theta^2\right]  .
\end{split}
\end{equation}
From \eqref{k12} and \eqref{eq:lower-bound-numerator} and  we can bound right-hand side of \eqref{bar-J} as follows, 
 \begin{equation} \label{eq:bound-int-theta23}
\begin{split}
&\ol J (d,\beta,N,T,\mathbf{R}_3) \\
&\leq  C N^2\iint_{\mathbf{R}_3}\frac{1}{t+s}
  \exp\left(-\frac{\frac{1}{2}\beta^{\frac{1}{2}}N^{\frac{1}{2}}\left[\big(t^{\frac{3}{4}}-s^{\frac{3}{4}}\big)^2 
      \right] }{2(t+s)}\right) \\
  & \quad  \times \int_{0}^{\pi}  \exp\left(-\frac{\beta^{\frac{1}{2}}N^{\frac{1}{2}}
     t^{\frac{3}{4}}s^{\frac{3}{4}}\theta^2 }{2(t+s)}\right)d\theta
         dsdt .
     \end{split}
\end{equation}
Next, we integrate inner integral on the right-hand side of \eqref{eq:bound-int-theta23} over $\theta$ using $s+t \leq 2t$ and \eqref{eq:est-int-exponential} to get
\begin{equation} \label{eq:bound-int-theta}
\begin{split}
\int_{0}^{\pi} \exp\left(-\frac{\beta^{\frac{1}{2}}N^{\frac{1}{2}}  s^{\frac{3}{4}}t^{\frac{3}{4}}\theta^2}{2t}\right)d\theta 
&\le C\beta^{-\frac{1}{4}}N^{-\frac{1}{4}}s^{-\frac{3}{8}}t^{\frac{1}{8}}. 
\end{split}
\end{equation}
 From \eqref{eq:bound-int-theta23} and \eqref{eq:bound-int-theta} and by using $s+t \geq t$ we get, 
  \begin{equation} \label{s-23}
\begin{split}
&\ol J (d,\beta,N,T,\mathbf{R}_3) \\
& \leq C\iint_{ \mathbf{R}_3}  \left(\beta^{-\frac{1}{4}}N^{-\frac{1}{4}} t^{\frac{1}{8}}s^{-\frac{3}{8}}\right)    \frac{1}{t} 
     \exp\left(-\frac{ C\beta^{\frac{1}{2}}N^{\frac{1}{2}}\big(t^{\frac{3}{4}}-s^{\frac{3}{4}}\big)^2 }{4t}\right)dsdt \\
  &    \leq C\iint_{\mathbf{R}_3 \cap \{s  \geq \eps t]\}}   \beta^{-\frac{1}{4}}N^{-\frac{1}{4}} t^{-\frac{7}{8}}s^{-\frac{3}{8}}   \exp\left(-\frac{ C\beta^{\frac{1}{2}}N^{\frac{1}{2}}\big(t^{\frac{3}{4}}-s^{\frac{3}{4}}\big)^2 }{4t}\right)dsdt \\
    &\quad +  C\iint_{\mathbf{R}_3 \cap \{s \in [0,\eps t]\}}  \beta^{-\frac{1}{4}}N^{-\frac{1}{4}} t^{-\frac{7}{8}}s^{-\frac{3}{8}} 
     \exp\left(-\frac{ C\beta^{\frac{1}{2}}N^{\frac{1}{2}}\big(t^{\frac{3}{4}}-s^{\frac{3}{4}}\big)^2 }{4t}\right)dsdt  \\
&     =: \ol J_1 (d,\beta,N,T,\mathbf{R}_3) +\ol J_2 (d,\beta,N,T,\mathbf{R}_3), 
 \end{split}
\end{equation}
 where $\eps \in (0,1)$. 

First we consider the case where $\eps  t \leq s \leq t$.  
Using the mean value theorem, we see that there exists $r=r(s,t)$ such that
\begin{equation} \label{mean-val} 
t^{3/4}-s^{3/4}= Cr^{-\frac{1}{4}}(t-s) \ge Ct^{-\frac{1}{4}}(t-s).
\end{equation}
Using $s^{-\frac{3}{8}} \leq  \eps^{-\frac{3}{8}}t^{-\frac{3}{8}} $, \eqref{R-def}, \eqref{mean-val} and \eqref{eq:est-int-exponential} on \eqref{s-23} we get
\begin{equation} \label{eq:large-st}
\begin{split}
&\ol J_1 (d,\beta,N,T,\mathbf{R}_3) \\
&\le C(\eps) \beta^{-\frac{1}{4}}N^{-\frac{1}{4}}\int_{0}^{4}t^{-\frac{5}{4}} \int_{0}^{t}  
     \exp\left(-\frac{C\beta^{\frac{1}{2}}N^{\frac{1}{2}}\big(t^{\frac{3}{4}}-s^{\frac{3}{4}}\big)^2 }{4t}\right)dsdt \\
     &\le  C(\eps) \beta^{-\frac{1}{4}}N^{-\frac{1}{4}}\int_{0}^{4}t^{-\frac{5}{4}}  \int_{0}^{t}  
     \exp\left(-\frac{C\beta^{\frac{1}{2}}N^{\frac{1}{2}}t^{-\frac{1}{2}}(t-s)^2 }{4t}\right)dsdt \\
        &\le  C(\eps) \beta^{-\frac{1}{4}}N^{-\frac{1}{4}}\int_{0}^{4} t^{-\frac{5}{4}}    \int_{-\infty}^{\infty}
     \exp\left(-\frac{1}{4}C\beta^{\frac{1}{2}}N^{\frac{1}{2}}t^{-\frac{3}{2}}(t-s)^2 \right)dsdt \\
             &\le C(\eps) \beta^{-\frac{1}{4}}N^{-\frac{1}{4}}\int_{0}^{4}t^{-\frac{5}{4}} \beta^{-\frac{1}{4}}N^{-\frac{1}{4}}t^{\frac{3}{4}}dt \\
&\le C(\eps)\beta^{-\frac{1}{2}}N^{-\frac{1}{2}}\int_{0}^{1}t^{-\frac{1}{2}}   dt \\
  &\le C(\eps)\beta^{-\frac{1}{2}}N^{-\frac{1}{2}}.
\end{split}
\end{equation} 
Next we consider the case of $0\leq s < \eps  t $. Then we have 
\begin{equation} \label{dd0}
\frac{\big(t^{\frac{3}{4}}-s^{\frac{3}{4}}\big)^2 }{4t}  \geq c_1(\eps) t^{\frac{1}{2}}. 
\end{equation} 
Using $0\leq s < \eps  t $ again, now together with \eqref{kap-al23}, \eqref{drift} and \eqref{d-r-1} we get 
\begin{equation*} 
\begin{split}
|D_t^{(k)}+D_s^{(\ell)}|^2
   & \le \big(|D_t|-|D_s|\big)^2 + 4|D_t|\cdot|D_s|    \\
&\le\beta^{1/2}N^{1/2}\left[\big(t^{3/4}-s^{3/4}\big)^2 
     + 4t^{3/4}s^{3/4}\right]  \\
     &\le( 1 +4\eps^{3/4})\beta^{1/2}N^{1/2}t^{3/2} .
\end{split}
\end{equation*}
Recall that on $\mathbf{R}_3$ we have $|D_t^{(k)}+D_s^{(\ell)}|^2 >8$ (see \eqref{R-def}), hence it follows that there exists $c_2(\eps)>0$ such that, 
\begin{equation}  \label{ dd2}
t  \geq c_2(\eps) \beta^{-1/3}N^{-1/3}.
\end{equation}
From \eqref{s-23}, \eqref{dd0} and \eqref{ dd2} we get  
\begin{equation} \label{eq:sm-st2}
\begin{split}
&\ol J_2 (d,\beta,N,T,\mathbf{R}_3) \\
   &\hspace{1cm} 
   \leq C\beta^{-\frac{1}{4}}N^{-\frac{1}{4}}  \int_{0}^{4}s^{-\frac{3}{8}}ds \int_{c_2(\eps) \beta^{-\frac{1}{3}}N^{-\frac{1}{3}}}^{4} t^{-\frac{7}{8}}     
     \exp\left(-C(\eps) \beta^{\frac{1}{2}}N^{\frac{1}{2}}t^{\frac{1}{2}}\right)dt \\  
  &\leq C\beta^{-\frac{1}{4}}N^{-\frac{1}{4}}  \exp\left(-C(\eps) \beta^{\frac{1}{6}}N^{\frac{1}{6}} \right)  \int_{0}^{4} t^{-\frac{7}{8}}     dt \\  
     &\le C\beta^{-\frac{1}{2}}N^{-\frac{1}{2}}.
\end{split}
\end{equation} 
Plugging in \eqref{eq:large-st} and \eqref{eq:sm-st2} into \eqref{s-23} and 
multiplying by $N^2$, we get $\beta^{-1/2}N^{3/2}$ as required.  This 
finishes the proof of Lemma \ref{lem-j-r3}. 
 
\end{proof} 

\begin{proof} [Proof of Proposition \ref{prop-j} for $J_2(d,\beta, N,T)$]  
The bounds on $J_2(d,\beta, N,T)$ for $d=2,3$ follow directly from Lemmas \ref{lem-j-r1}--\ref{lem-j-r3}.
\end{proof} 

\section{Self-intersection occupation measure for $d=2,3$} \label{sec-self-2-3}
In this section we derive the bounds on $J_{d,1}(\beta, N,T)$ from Proposition \ref{prop-j} for $d=2,3$. 
Recall that in this case we are dealing with the occupation measure terms related to a single 
branch which intersects with itself, therefore there are $N$ such contributions.  

We recall that $f^{(k,k,\alpha)}_{t,s}$ is the probability density function of 
$B^{(k)}_t-B^{(k)}_s$ under $P^\lambda_T$. Instead of working with $P^\lambda_T$, we will work with $P_T$ and consider 
processes $B^{(k)}_t+D^{(k)}_t$, where 
%$D^{(k)},D^{(\ell)}$ are the respective drift processes.  
$D^{(k)}$ is the corresponding drift processes.  

Assume that $s<t$.  First note that
\begin{equation*}
B^{(k)}_t-B^{(k)}_s \sim \mathcal N(0,t-s),
\end{equation*}
where as before $\mathcal N(\cdot,\cdot)$ stands for the $d$-dimensional normal distribution.

Recall that the drift magnitudes are given by \eqref{drift}. Without loss of 
generality, we can assume that $D^{(k)}$ points in the $\mathbf{e}_1$
direction, where $\mathbf{e}_1$ is the first coordinate vector. 

For simplicity we define 
\begin{equation} \label{zeta} 
\zeta=\beta^{\frac{1}{4}}N^{\frac{1}{4}},
\end{equation}
then using \eqref{drift} and \eqref{kap-al23} we get 
\begin{equation} \label{dens-felf} 
\begin{split}
f^{(k,k,\alpha)}_{t,s}(z) &= \left(2\pi(t-s)\right)^{-d/2}
 \exp\left(-\frac{|z-(\zeta t^{\frac{3}{4}}-\zeta s^{\frac{3}{4}})
     \mathbf{e}_1|^2}{2(t-s)}\right) \\
&= \left(2\pi(t-s)\right)^{-d/2}
  \exp\left(-\frac{\big(z_1-\zeta t^{\frac{3}{4}}+\zeta s^{\frac{3}{4}}\big)^2}{2(t-s)}\right) \\
&\hspace{4cm}  \times\prod_{k=2}^{d}\exp\left(-\frac{z_k^2}{2(t-s)}\right). 
\end{split}
\end{equation}
 For convenience we again split the area of integration in $J_1(\beta,N,T)$, $\{(s,t): \, 0\leq s \leq t \leq T\}$ to the following subregions. 
\begin{equation} \label{r-i-j1} 
\begin{aligned} 
\hat {\textbf{R}}_1& = \{(s,t) \in [0,T]^2 \, : t^{\frac{3}{4}}-s^{\frac{3}{4}}  > C_{(\ref{r-i-j1})}\zeta^{-1}\}, \\
\hat{\mathbf{R}}_2 &= \{ (t,s) \in [0,T]^2 : t^{\frac{3}{4}}-s^{\frac{3}{4}}
                   < C_{(\ref{r-i-j1})}\zeta^{-1}   \}, 
\end{aligned} 
\end{equation}
where $C_{(\ref{r-i-j1})}>0$ is a constant to be specified later. 

Define  
\begin{equation} \label{hat-j} 
\hat J (d,\beta,N,T, \hat{\mathbf{R}}_i)  = N  \iint_{\hat {\textbf{R}}_i}\int_{\mathbf{B}_2(0)} f^{(1,1,\alpha)}_{t,s}(z) dz dsdt. 
\end{equation}

Since the self-occupation measure is similar for all branches of the polymers it follows from \eqref{eq:L-squared}, 
\begin{equation} \label{j-1-dec} 
 J_1 (d,\beta,N,T) \leq    C\sum_{i=1}^2 \hat J (d,\beta,N,T, \hat{\mathbf{R}}_i). 
\end{equation} 
We introduce a few technical lemmas that will help us to bound $J_1(d,\beta,N,T)$. 
 
\begin{lemma} \label{lem-j1-1} 
For $d=2,3$ there exist positive constants $C_{(\ref{r-i-j1})}$ and $C_{\ref{lem-j1-1}}$ such that the following bound holds for 
$1\leq T \leq N, \,  \beta \geq 1, \, N \in \mathbb Z^+$: 
\begin{equation*}
\hat J (d,\beta,N,T,\hat{\mathbf{R}}_1)  \leq C_{\ref{lem-j1-1}}\beta^{-\frac{1}{2}} N^{\frac{3}{2}} T^{\frac{1}{2}}\log(\beta T).
\end{equation*}
\end{lemma} 
\begin{proof}  
From \eqref{r-i-j1} we can choose $C_{(\ref{r-i-j1})}$ in the definition of $\hat{\mathbf{R}}_1$ such that 
\begin{equation} \label{t1}  
(z_1-\zeta t^{\frac{3}{4}}+\zeta s^{\frac{3}{4}})^2 \geq \frac{1}{2}(\zeta t^{3/4}-\zeta s^{3/4})^2, \quad \textrm{for all }t,s \in \hat{\mathbf{R}}_1. 
\end{equation} 
From \eqref{dens-felf}, \eqref{hat-j} and \eqref{t1} it follows that 
\begin{equation} \label{b-bnd1} 
\begin{split}
\hat J (d,\beta,N,T,\hat{\mathbf{R}}_1)   
&\leq N \iint_{\hat{\mathbf{R}}_1} \frac{1}{\sqrt{2\pi (t-s)}}  
\int_{|z_1|<2} e^{-\frac{\big(z_1-\zeta t^{\frac{3}{d+2}} +\zeta s^{\frac{3}{4}}\big)^2}{2(t-s)}}dz_1  \\
&\hspace{2cm} \times\prod_{k=2}^{d} \left(\int_{|z_k|<2}  \frac{1}{\sqrt{2\pi (t-s)}} 
     e^{-\frac{z_k^2}{2(t-s)}}dz_k\right) ds dt\\ 
&\leq N \iint_{\hat{\mathbf{R}}_1}  \frac{1}{\sqrt{2\pi (t-s)}}  
\int_{|z_1|<2} e^{-\frac{\big(\zeta t^{\frac{3}{4}} -\zeta s^{\frac{3}{4}}\big)^2}{4(t-s)}}dz_1  \\
&\hspace{2cm} \times\prod_{k=2}^{d} \left(\int_{z_k\in\mathbf{R}}  \frac{1}{\sqrt{2\pi (t-s)}} 
     e^{-\frac{z_k^2}{2(t-s)}}dz_k\right) ds dt\\ 
&\leq N \iint_{\hat{\mathbf{R}}_1}\frac{C}{\sqrt{2\pi(t-s)}}
    \exp\left(-\frac{\zeta^2\big(t^{\frac{3}{4}}-s^{\frac{3}{4}}\big)^2}{4(t-s)}\right)ds dt. 
\end{split}
\end{equation}
Since $0<s<t$, the mean value theorem implies that for some $r\in(s,t)$ we have
\begin{equation*}
\begin{split}
\frac{\big(t^{\frac{3}{4}}-s^{\frac{3}{4}}\big)^2}{t-s}
&= \left(\frac{t^{\frac{3}{4}}-s^{\frac{3}{4}}}{t-s}\right)^2(t-s)
= \big(r^{-\frac{1}{4}}\big)^2(t-s)  \\
&\ge  t^{-\frac{1}{2}} (t-s), 
\end{split}
\end{equation*}
where in the last line follows because $r<t$. 

We therefore have 
\begin{equation*}
\begin{split}
\hat J (d,\beta,N,T,\hat{\mathbf{R}}_1)  
\le N\iint_{\hat{\mathbf{R}}_1} \frac{C}{\sqrt{2\pi(t-s)}}
    \exp\left(-\frac{1}{4}t^{-\frac{1}{2}}(t-s)\right)ds dt.
\end{split}
\end{equation*}
Using the fact that $0\leq s<t \leq T$ and
\begin{equation*}
\int_{\mathbf{B}_2(0)} f^{(k,k,\alpha)}_{t,s}(z)dz \leq 1, 
\end{equation*}
together with \eqref{hat-j} we arrive at 
\begin{equation} \label{t3} 
\begin{aligned} 
&\hat J (d,\beta,N,T,\hat{\mathbf{R}}_1)   \\ 
&
\le N \iint_{\hat{\mathbf{R}}_1} \left( \left( \frac{C}{\sqrt{2\pi(t-s)}}
    \exp\left(-\frac{1}{4}T^{-\frac{1}{2}}\zeta^2(t-s)\right)  \right) \wedge  1  \right)  ds dt . 
\end{aligned} 
\end{equation}
Define 
%\begin{align} \label{gammas} 
% \gamma_1 = \frac{5}{4}, \qquad  \gamma_2 = \frac{3}{4}, 
%\end{align}
%and  
\begin{equation} \label{t5} 
\gamma :=  \frac{1}{4}(N \beta)^{-\frac{1}{2}} T^{\frac{1}{2}} \log (  T^{\frac{5}{4}} \beta^{\frac{3}{4}}). 
\end{equation}
 From \eqref{zeta} and \eqref{t5} we get for all $t-s >\gamma$, 
\begin{equation} \label{eq:estimate-density}
\begin{split}
& \frac{1}{\sqrt{2\pi(t-s)}} 
    \exp\left(- \frac{1}{4}T^{-\frac{1}{2}}\zeta^2(t-s) \right)  \\
&\leq C\gamma^{- 1/2} \exp\left(- \frac{1}{4}T^{-\frac{2d-2}{d+2}} \zeta^2 \gamma\right) \\
&    \leq  C\gamma^{- 1/2} \exp\left(- \frac{1}{4}T^{-\frac{1}{2}} N^{\frac{1}{2}} 
             \beta^{\frac{1}{2}} \gamma\right)  \\
&    \leq C (N \beta)^{\frac{1}{4}} T^{-\frac{1}{4}}  
        \log ( T^{\frac{5}{4}} \beta^{\frac{3}{4}})^{-1/2}
            (  T^{\frac{5}{4}} \beta^{\frac{3}{4}})^{-1} \\
&     \leq C\beta^{-\frac{1}{2}} N^{\frac{1}{2}} T^{-\frac{3}{2}}
        (\log (\beta T ))^{-1/2}. 
\end{split}
\end{equation} 
%where we have plugged in the values of $\gamma_i$ from \eqref{gammas} in the last inequality.   

From \eqref{t3}, \eqref{eq:estimate-density} and since $0<s<t<T$ it 
follows that 
\begin{equation*}
\begin{split}
&\hat J (d,\beta,N,T,\hat{\mathbf{R}}_1)   \\ 
&
\le CN\int_{0}^T  \int_{0}^{t-\gamma} \frac{1}{\sqrt{2\pi(t-s)}}
    \exp\left(-\frac{1}{4}t^{-\frac{1}{2}}(t-s)\right)ds dt \\ 
&\quad + CN\int_{0}^T  \int_{t-\gamma}^{t}1 ds dt \\ 
&
\le CN\int_{0}^T  \int_{0}^{T} \beta^{-\frac{1}{2}} N^{\frac{1}{2}} T^{-\frac{3}{2}}
        (\log (\beta T ))^{-1/2} ds dt + CNT\gamma \\
&
\le C \beta^{-\frac{1}{2}} N^{\frac{3}{2}} T^{\frac{1}{2}}
        (\log (\beta T ))^{-1/2} 
  + CNT(N \beta)^{-\frac{1}{2}} T^{\frac{1}{2}} \log (  T  \beta)\\
&\le C \beta^{-\frac{1}{2}} N^{\frac{3}{2}} T^{\frac{1}{2}}
        (\log (\beta T ))^{-1/2} 
  + C\beta^{-\frac{1}{2}} N^{\frac{1}{2}}T^{\frac{3}{2}} \log (  T  \beta) .
\end{split}
\end{equation*}
By choosing 
\begin{equation} \label{eq:T-bound-1}
T\le N, 
\end{equation}
we get
\begin{equation*}
\begin{split}
\hat J (d,\beta,N,T,\hat{\mathbf{R}}_1)  \le C \beta^{-\frac{1}{2}} N^{\frac{3}{2}} T^{\frac{1}{2}}
        \log (\beta T ),  \quad \textrm{for all } T,\beta \geq 1, 
\end{split}
\end{equation*}
This completes the proof of Lemma \ref{lem-j1-1}.  
\end{proof} 

\begin{lemma} \label{lem-j1-2} 
For $d=2,3$ there exists a constant $C_{\ref{lem-j1-2} }>0$ such that the following bound holds for $ \beta  \geq 1$, $T>0$ and $N\in \mathbb{Z}^+$:
\begin{equation*}
\hat J (d,\beta,N,T,\hat{\mathbf{R}}_2)  \leq C_{\ref{lem-j1-2} } \beta^{-2/3} N^{1/3}.
\end{equation*}
\end{lemma} 

\begin{proof} 
From \eqref{zeta} and \eqref{r-i-j1} it follows that on $\hat{\mathbf{R}}_2$ we have 
\begin{equation} \label{ee1} 
\begin{aligned} 
s^{\frac{3}{d+2}}  &\leq t^{\frac{3}{d+2}} 
   \leq  C\beta^{-\frac{1}{d+2}}N^{-\frac{1}{d+2}} 
           + s^{\frac{3}{d+2}},  \\
s  &\leq C\beta^{-\frac{1}{3}}N^{-\frac{1}{3}}. 
\end{aligned} 
\end{equation}
By integrating over $z$ in the right-hand side of \eqref{hat-j} and then using \eqref{ee1} we get  
\begin{equation}  \label{case2-int} 
\begin{aligned}  
\hat J(d,\beta,N,T,\hat{\mathbf{R}}_2) &\leq N |\hat{\mathbf{R}}_2| \\ 
  &  \leq  N\int_{0}^{C\beta^{-\frac{1}{3}}N^{-\frac{1}{3}}} \int_{0}^{ (C\beta^{-\frac{1}{4}}N^{-\frac{1}{4}} 
            + s^{\frac{3}{4}})^{\frac{4}{3}} } dt ds  \\
   & \leq   N\int_{0}^{C\beta^{-\frac{1}{3}}N^{-\frac{1}{3}}} 
       ( C\beta^{-\frac{1}{4}}N^{-\frac{1}{4}} 
             + s^{\frac{3}{4}})^{\frac{4}{3}} ds \\
             & \leq C N\int_{0}^{C\beta^{-\frac{1}{3}}N^{-\frac{1}{3}}} (\beta^{-\frac{1}{4}}N^{-\frac{1}{4}})^{\frac{4}{3}} ds   \\ 
             &\leq C \beta^{-\frac{2}{3}}N^{\frac{1}{3}}. 
 \end{aligned}
\end{equation}
This completes the proof of Lemma \ref{lem-j1-2}. 
\end{proof}

\begin{proof} [Proof of Proposition \ref{prop-j} for $J_1(d,\beta, N,T)$]  
The bounds on $J_1(d,\beta, N,T)$ for $d=2,3$ follow directly from Lemmas \ref{lem-j1-1} and \ref{lem-j1-2} and from \eqref{j-1-dec}.
\end{proof} 

 \section{Proof of Proposition \ref{prop-i-2} for $d=1$} \label{sec-local-1}
Following \eqref{kap-al23} we choose 
\begin{equation}\label{kap-al1} 
\kappa = \beta^{1/3} N^{1/3}, \quad \alpha =0, 
\end{equation}
for the parameters of drift magnitude given by \eqref{drift}. Note that in 
the one-dimensional case we give all $N$ particles drift in the positive 
direction of $\mathbb{R}$. 
 
\subsection{Cross-intersection occupation measure} 
We follow the argument in Section \ref{sec-cross-2-3}, with the 
change that under $\mathbb{P}_T^\lambda$ all drifts go in the positive 
direction on $\mathbb{R}$.  Recall that there are $O(N^2)$ pairs of 
particles.  We therefore have as in \eqref{eq:L-squared} as follows, 
\begin{equation}  \label{fg11} 
\begin{aligned} 
&J_2(1,\beta, N,T) 
:= CN^2 \int_0^T\int_0^T \int_{\mathbf{B}_2(0)} f^{(k,\ell,\alpha)}_{t,s}(z) dz ds dt \\
&\leq CN^2 \int_0^T\int_0^T \left( \left( \frac{1}{(t+s)^{1/2} } \int_{|z_1|<2} e^{-\frac{(z_1-\beta^{1/3} N^{1/3}( t  - s))^2}{2(t+s)}}dz_1 \right) \wedge 1 \right)ds dt.
\end{aligned}
\end{equation} 

The following lemma give us the essential bound on $J_2(1,\beta, N,T)$. 
\begin{lemma} \label{dim1-lem-j2}
There exists a constant $C_{\ref{dim1-lem-j2}}>0$ such that for all  $\beta, T \geq 1$ and $N\in\mathbb{Z}^+$ the following bound holds: 
\begin{equation*}
 J_2(1,\beta, N,T) \leq C_{\ref{dim1-lem-j2}}\beta^{2/3} N^{5/3}T.
 \end{equation*}
\end{lemma} 
\begin{proof} 
From \eqref{fg11} we have , 
\begin{equation}  \label{fg21} 
\begin{aligned} 
&J_2(1,\beta, N,T) \\
&\leq CN^2 \int_0^T\int_0^T \left( \left( \frac{1}{(t+s)^{1/2}}  \int_{|z_1|<2} e^{-\frac{(z_1-\beta^{1/3} N^{1/3}( t  - s ))^2}{2(t+s)}}dz_1 \right) \wedge 1 \right)ds dt. 
\end{aligned}
\end{equation} 

Note that, \begin{equation} \label{c-choice} 
(2-\beta^{1/3} N^{1/3}(t-s))^2 \geq \frac{1}{2}\beta^{2/3} N^{2/3}(t-s)^2, \quad  \textrm{for all } t-s >8\beta^{-\frac{1}{3}} N^{-\frac{1}{3}},
\end{equation}
and consider the following regions: 
\begin{equation}  \label{regions-1dim} 
 \begin{aligned}
\hat{\mathbf{R}}_1 &= \{ (t,s) \in [0,T]^2: \beta^{1/3} N^{1/3}(t-s) > 8\}  \\
\hat{\mathbf{R}}_2 &= \{ (t,s) \in [0,T]^2: \beta^{1/3} N^{1/3}(t-s) \leq  8\}. 
\end{aligned} 
\end{equation} 
We define  
\begin{equation} \label{dim1-j2-def} 
\tilde{J}_2(1,\beta,N,T,\hat{\mathbf{R}}_i) := N^2 \iint_{\hat{\mathbf{R}}_i} \int_{\mathbf{B}_2(0)} f^{(1,1,\alpha)}_{t,s}(z)dz  ds dt, \quad i=1,2, 
\end{equation}
and note that 
\begin{equation} \label{fg432} 
J_2(1,\beta,N,T) \leq \sum_{i=1}^2 \tilde{J}_2(1,\beta,N,T,\hat{\mathbf{R}}_i ). 
\end{equation} 
From \eqref{fg21}, \eqref{c-choice}, \eqref{regions-1dim} and \eqref{dim1-j2-def} we get that 
\begin{equation}  \label{fg31} 
\begin{aligned} 
&\tilde{J}_2(1,\beta,N,T,\hat{\mathbf{R}}_1)\\
&\leq C N^2  \int_{0}^T  \int_{0}^t \left ( \frac{1}{(2\pi (t+s))^{1/2}}  e^{-\beta^{2/3} N^{2/3} \frac{(t-s)^2}{4(t+s)}} \right) \wedge 1 ds dt  \\
&\leq C  N^2\int  \int_{t-s>\delta}  \left ( \frac{1}{(2\pi (t+s))^{1/2}}  e^{-\beta^{2/3} N^{2/3} \frac{(t-s)^2}{4(t+s)}} \right) \wedge 1 ds dt  \\
&\quad + CN^2  \int  \int_{0<t-s \leq \delta}  \left ( \frac{1}{(2\pi (t+s))^{1/2}}  e^{-\beta^{2/3} N^{2/3} \frac{(t-s)^2}{4(t+s)}} \right) \wedge 1 ds dt  \\
&\leq C N^2 \int  \int_{t-s > \delta } \frac{1}{(2\pi (t+s))^{1/2}} e^{-\beta^{2/3} N^{2/3} \frac{(t-s)^2}{4(t+s)}}  ds dt   +C N^2\int_0^T\int_{t - \delta}^t  ds dt \\
&=:H_1(\beta,N,T) + H_2(\beta,N,T),
\end{aligned}
\end{equation}
where we define 
\begin{equation} \label{dl-1d} 
\delta =\beta^{-1/3} N^{-1/3}.
\end{equation} 
First we note that
\begin{equation} \label{b-est} 
 H_2(\beta,N,T) \le CN^2T\delta \leq \beta^{-1/3} N^{5/3}T.
\end{equation}
Let us change variables to 
\begin{align*}
a = t-s,  && b = t + s 
\end{align*}
and absorb the Jacobian determinant into the constant $C$. We find
\begin{equation} \label{eq:B-est}
\begin{split}
 H_1(\beta,N,T)&\le CN^2\int_{0}^{2T} \frac{1}{(2\pi b)^{1/2}} \left(\int_{\delta}^{\infty}
    e^{-\beta^{2/3} N^{2/3} \frac{a^2}{4b}}  da \right) db.
\end{split}
\end{equation}
Dealing with the inner integral, we use the bound on the integral over the 
Gaussian density to get,
\begin{equation} \label{new-b}
\begin{split}
\frac{1}{(2\pi b)^{1/2}} & \int_{\delta}^{\infty}e^{-\beta^{2/3} N^{2/3} 
    \frac{a^2}{4b}}  da  \\
& \leq
\frac{N^{-1/3}\beta^{-1/3}}{(2\pi b N^{-2/3}\beta^{-2/3})^{1/2}}  \int_{-\infty}^{\infty}e^{- \frac{a^2}{4bN^{-2/3}\beta^{-2/3} }}  da  \\ 
 &\leq CN^{-1/3}\beta^{-1/3}. 
\end{split}
\end{equation}
Applying \eqref{new-b} to the inner integral in \eqref{eq:B-est} with we get
\begin{equation} \label{eq:G-est}
\begin{split}
 H_1(\beta,N,T) &\le CN^2\int_{0}^{2T} CN^{-1/3}\beta^{-1/3} db\\
&\le C\beta^{-1/3}N^{5/3} T^{}. 
\end{split}
\end{equation}
 
Plugging-in our estimates from \eqref{b-est} and \eqref{eq:G-est} to \eqref{fg31} we get
\begin{equation}\label{fg332} 
 \tilde{J}_2(1,\beta,N,T,\hat{\mathbf{R}}_1)\leq C\beta^{-1/3}N^{5/3} T^{}. 
 \end{equation}
From \eqref{regions-1dim} and \eqref{dim1-j2-def} it follows that 
\begin{equation}  \label{fg32} 
\begin{aligned} 
 \tilde{J}_2(\beta,N,T,\hat{\mathbf{R}}_2) &\leq N^2 |\hat{\mathbf{R}}_2 |   \\
&\leq C N^2 \beta^{-1/3} N^{-1/3} T \\ 
&\leq C   \beta^{-1/3} N^{5/3} T. 
\end{aligned}
\end{equation} 
 From \eqref{fg432}, \eqref{fg332} and \eqref{fg32} we get Lemma \ref{dim1-lem-j2}.  
\end{proof} 
\subsection{Self-intersection occupation measure}  
Using \eqref{drift} and \eqref{kap-al1} we deduce by following similar lines as in \eqref{dens-felf}, \eqref{r-i-j1} and \eqref{j-1-dec}, we need to  bound the following integral: 
\begin{equation} \label{b-bnd41} 
\begin{split}
J_1&(1,\beta,N,T) := N\int_0^T\int_0^T \int_{\mathbf{B}_2(0)} f^{(1,1,\alpha)}_{t,s}(z)dz  ds dt \\ 
&=N\int_0^T\int_0^T \frac{1}{(2\pi (t-s))^{1/2}}  \int_{|z|<2} e^{-\frac{(z-\beta^{1/3} N^{1/3}(t-s))^2}{2(t-s)}}dzds dt.   
\end{split}
\end{equation}
In the following lemma we derive the bound on $ J_1(\beta, N,T)$. 
\begin{lemma}  \label{dim1-lem-j1}
There exists a constant $C_{\ref{dim1-lem-j1}}>0$ such that for all  $\beta, T \geq 1$ and $N\in\mathbb{Z}^+$ the following bound holds: 
\begin{equation*}
 J_1(1,\beta, N,T) \leq C_{\ref{dim1-lem-j1}}\beta^{2/3} N^{2/3}T.
\end{equation*}
\end{lemma} 
\begin{proof} 
Recalling $\hat{\mathbf{R}}_i$ defined in \eqref{regions-1dim} we denote 
\begin{equation} \label{dim1-j1-def} 
\tilde{J}_1(1,\beta,N,T,\hat{\mathbf{R}}_i) := N \iint_{\hat{\mathbf{R}}_i} \int_{\mathbf{B}_2(0)} f^{(1,1,\alpha)}_{t,s}(z)dz  ds dt, \quad i=1,2, 
\end{equation}
and note that 
\begin{equation*}
\tilde{J}_1(1,\beta,N,T) \leq \sum_{i=1}^2 \tilde{J}_1(1,\beta,N,T,\hat{\mathbf{R}}_i ). 
\end{equation*}
Using the fact that $\int_{\mathbf{B}_2(0)} f^{(1,1,\alpha)}_{t,s}(z)dz  \leq 1$ and \eqref{c-choice}, we get for $\delta$ as in \eqref{dl-1d}, 
\begin{equation} \label{b-bnd41a} 
\begin{split}
\tilde{J}_1(1,\beta, & N,T,\mathbf{R_1} )   \\
& \leq CN\int \int_{\hat{\mathbf{R}}_1 \cap \{t-s>\delta\}  } \frac{1}{2\pi (t-s)}  e^{-\frac{ \beta^{2/3} N^{2/3}(t-s)^2}{4(t-s)}}ds dt \\ 
&\quad  + CN\int \int_{\hat{\mathbf{R}}_1 \cap \{t-s \leq \delta\}  } ds dt.
 \end{split}
\end{equation}
Note that this integral is similar to \eqref{fg31} (with a difference by a factor of $N$) hence we get the bound  
 \begin{equation} \label{hj1} 
\tilde{J}_1(1,\beta,N,T,\hat{\mathbf{R}}_1 ) \leq C\beta^{2/3}N^{2/3} T. 
\end{equation}
Using again the trivial bound on the integral over $f^{(1,1,\alpha)}$ we get from \eqref{dim1-j1-def} and \eqref{regions-1dim} that
\begin{equation} \label{hj2} 
\begin{split}
\tilde{J}_1(1,\beta,N,T,\hat{\mathbf{R}}_2)  & \leq C  \iint_{\hat{\mathbf{R}}_2 }ds dt \leq C\beta^{-1/3} N^{2/3} T, 
 \end{split}
\end{equation}
where we used a similar bound as in \eqref{fg32} with a difference of a factor of $N$ in the last inequality. 

From \eqref{hj1},  \eqref{hj2} and \eqref{dim1-j1-def} we get the result of Lemma \ref{dim1-lem-j1}.
\end{proof} 

Now we are ready to prove Theorem \ref{prop-z} for $d=1$. 
\begin{proof} [Proof of Proposition \ref{prop-i-2}  for $d=1$] 
The proof follows similar lines to the proof in the case where $d=2,3$ only we now use Lemmas \ref{dim1-lem-j2} and \ref{dim1-lem-j1} instead of Proposition \ref{prop-j}. 
\end{proof} 

\begin{appendix}

\section{Proof of Lemma \ref{lemma-ber}}\label{secA1}

\begin{proof}
Such large deviations results are well-known, but they are usually 
expressed in asymptotic form.  We need an upper bound for the probability, 
and we give a self-contained proof.

Using Markov's inequality and the moment generating function of the 
binomial distribution, and
choosing $t>0$, for $\alpha\in(p,1)$ we get
\begin{equation}  \label{eq:large-dev}
P\left(S_{n}>\alpha n\right)
 \le \frac{E\left[\exp\left(tS_{n}\right)\right]}{\exp\left(\alpha tn\right)}
 = \left(e^{-\alpha t}\left[q+pe^t\right]\right)^{n} .
\end{equation}
Let $f(t)=e^{-\alpha t}[q+pe^t]$. 
We find that $f(t)\to\infty$ as $t\to\infty$. Since $\alpha>p$, we have
\begin{equation}  \label{eq:value-t-star}
f'(0)=p-\alpha <0.
\end{equation}
Thus $f$ achieves its minimum over $[0,\infty)$ in $(0,\infty)$.  Let $t^*$ 
be the infimum of those $t\in(0,\infty)$ for which the minimum is achieved.  

Solving $f'(t^*)=0$, we find $-\alpha q+(1-\alpha)pe^{t^*} = 0$ and so
\begin{equation*}
t^* = \log\frac{\alpha q}{(1-\alpha)p}.
\end{equation*}
Combining \eqref{eq:large-dev} and \eqref{eq:value-t-star} and substituting
$t=t^*$ gives
\begin{equation*}
P(S_n>\alpha n) \le \left(\frac{(1-\alpha)p}{\alpha q}\right)^{\alpha n}
   \left[q+p\frac{\alpha q}{(1-\alpha)p}\right]^n.  
\end{equation*}
This proves Lemma \ref{lemma-ber}.
\end{proof}

\end{appendix}

%\bibliographystyle{plain}
%\printindex
%\bibliography{polymer.bib}

\end{document}